\documentclass[a4paper, 11pt]{amsart}   
\usepackage{mathptmx, amssymb,amscd,latexsym, eulervm}   
\usepackage{amsmath}
\usepackage{amsthm}
\usepackage{mathdots}
\usepackage[pagebackref,colorlinks=true,linkcolor=blue,urlcolor=blue]{hyperref}
\usepackage{color}
\usepackage[onehalfspacing]{setspace}
\usepackage{tabularx}
\usepackage{amsfonts}
\usepackage{paralist}
\usepackage{aliascnt}
\usepackage[initials]{amsrefs}
\usepackage{amscd}
\usepackage{blkarray}
\usepackage{mathbbol}
\usepackage{setspace}
\usepackage[inner=2.3cm,outer=2.3cm, bottom=3.3cm]{geometry}
\usepackage{tikz, tikz-cd}

\usetikzlibrary{matrix}
\usetikzlibrary{arrows}

\BibSpec{collection.article}{%
	+{}  {\PrintAuthors}                {author}
	+{,} { \textit}                     {title}
	+{.} { }                            {part}
	+{:} { \textit}                     {subtitle}
	+{,} { \PrintContributions}         {contribution}
	+{,} { \PrintConference}            {conference}
	+{}  {\PrintBook}                   {book}
	+{,} { }                            {booktitle}
	+{,} { }                            {series}
	+{, vol.} { }                            {volume}
	+{,} { }                            {publisher}
	+{,} { \PrintDateB}                 {date}
	+{,} { pp.~}                        {pages}
	+{,} { }                            {status}
	+{,} { \PrintDOI}                   {doi}
	+{,} { available at \eprint}        {eprint}
	+{}  { \parenthesize}               {language}
	+{}  { \PrintTranslation}           {translation}
	+{;} { \PrintReprint}               {reprint}
	+{.} { }                            {note}
	+{.} {}                             {transition}
	+{}  {\SentenceSpace \PrintReviews} {review}
}

\newcommand{\kk}{\mathbb{k}}
\newcommand{\KK}{\mathbb{K}}

\newcommand{\NN}{\normalfont\mathbb{N}}

\newcommand{\xx}{{\normalfont\mathbf{x}}}
\newcommand{\yy}{\normalfont\mathbf{y}}

\newcommand{\mm}{{\normalfont\mathfrak{m}}}

\newcommand{\pp}{{\normalfont\mathfrak{p}}}

\newcommand{\nn}{{\normalfont\mathfrak{n}}}

\newcommand{\Ext}{\normalfont\text{Ext}}

\newcommand{\Ker}{\normalfont\text{Ker}}

\newcommand{\Quot}{\normalfont\text{Quot}}

\newcommand{\HT}{\normalfont\text{ht}}

\newcommand{\Ass}{\normalfont\text{Ass}}

\newcommand{\Hom}{\normalfont\text{Hom}}

\newcommand{\leng}{{\normalfont\text{length}}}

\newcommand{\BB}{\mathbb{B}}

\newcommand{\EE}{\mathbb{E}}
\newcommand{\VV}{\mathbb{V}}
\newcommand{\WW}{\mathbb{W}}

\newcommand{\LL}{\mathbb{L}}

\newcommand{\FF}{\mathbb{F}}

\newcommand{\HH}{\normalfont\text{H}}

\newcommand{\gr}{\normalfont\text{gr}}
\newcommand{\Sol}{\normalfont\text{Sol}}

\newcommand{\AAA}{\mathfrak{A}}

\newcommand{\Spec}{{\normalfont\text{Spec}}}

\newcommand{\Diff}{{\normalfont\text{Diff}}}

\makeatletter
\newcommand{\verbatimfont}[1]{\renewcommand{\verbatim@font}{\ttfamily#1}}
\makeatother

\newcommand{\m}{\mathfrak{m}}

\newcommand{\amult}{{\normalfont \text{amult}}}
\newcommand{\DiffR}{\Diff_{R/\kk}}
\newcommand{\Diag}{\Delta_{R/\kk}}
\newcommand{\Princ}{P_{R/\kk}}

\DeclareMathOperator{\ann}{ann}

\DeclareMathOperator{\hull}{hull}

\newtheorem{theorem}{Theorem}[section]

\newaliascnt{headcor}{headthm}

\aliascntresetthe{headcor}

\newaliascnt{headconj}{headthm}

\aliascntresetthe{headconj}

\newaliascnt{corollary}{theorem}
\newtheorem{corollary}[corollary]{Corollary}
\aliascntresetthe{corollary}

\newaliascnt{lemma}{theorem}
\newtheorem{lemma}[lemma]{Lemma}
\aliascntresetthe{lemma}

\newaliascnt{conjecture}{theorem}

\aliascntresetthe{conjecture}

\newaliascnt{proposition}{theorem}
\newtheorem{proposition}[proposition]{Proposition}
\aliascntresetthe{proposition}

\theoremstyle{definition}
\newaliascnt{definition}{theorem}
\newtheorem{definition}[definition]{Definition}
\aliascntresetthe{definition}

\newaliascnt{notation}{theorem}
\newtheorem{notation}[notation]{Notation}
\aliascntresetthe{notation}

\newaliascnt{example}{theorem}
\newtheorem{example}[example]{Example}
\aliascntresetthe{example}

\newaliascnt{examples}{theorem}

\aliascntresetthe{examples}

\newaliascnt{remark}{theorem}
\newtheorem{remark}[remark]{Remark}
\aliascntresetthe{remark}

\newaliascnt{problem}{theorem}

\aliascntresetthe{problem}

\newaliascnt{construction}{theorem}

\aliascntresetthe{construction}

\newaliascnt{setup}{theorem}
\newtheorem{setup}[setup]{Setup}
\aliascntresetthe{setup}

\newaliascnt{algorithm}{theorem}
\newtheorem{algorithm}[algorithm]{Algorithm}
\aliascntresetthe{algorithm}

\newaliascnt{observation}{theorem}

\aliascntresetthe{observation}

\newaliascnt{defprop}{theorem}

\aliascntresetthe{defprop}

\def\equationautorefname~#1\null{(#1)\null}
\def\sectionautorefname~#1\null{Section #1\null}
\def\subsectionautorefname~#1\null{\S #1\null}

\begin{document}

\title[Primary decomposition of modules: a computational differential approach]{Primary decomposition of modules: \\ a computational differential approach}

\author{Justin Chen}
\address{School of Mathematics, Georgia Institute of Technology,
	Atlanta, Georgia}
\email{justin.chen@math.gatech.edu}

\author{Yairon Cid-Ruiz}
\address{Department of Mathematics: Algebra and Geometry, Ghent University, Krijgslaan 281 – S25, 9000 Gent, Belgium}
\email{Yairon.CidRuiz@UGent.be}


\begin{abstract}
We study primary submodules and primary decompositions from a differential and computational point of view.
Our main theoretical contribution is  a general structure theory and a representation theorem for primary submodules of an arbitrary finitely generated module over a polynomial ring.
We characterize primary submodules in terms of differential operators and punctual Quot schemes.
Moreover, we introduce and implement an algorithm that computes a minimal differential primary decomposition for a module. 	
\end{abstract}

\maketitle

\section{Introduction}

The existence of primary decompositions has long been known, since the classical works of Lasker \cite{lasker1905theorie} and Noether \cite{noether1921idealtheorie}: over a Noetherian commutative ring, every proper submodule of a finitely generated module can be expressed as a finite intersection of primary submodules.
Accordingly, one can view primary submodules as the basic building blocks for arbitrary modules.
In this paper, we study the central notions of primary submodules and primary decompositions from a differential and computational point of view.

Let $\kk$ be a field of characteristic zero and $R$ a polynomial ring $R = \kk[x_1,\ldots,x_n]$.
The main objective of this paper is to characterize primary $R$-submodules with the use of \emph{differential operators} and \emph{punctual Quot schemes}.
We achieve this goal in \autoref{thm_main} (see also \autoref{cor_main}), which can be seen as an extension from ideals to $R$-modules of the representation theorem given in \cite{PRIM_IDEALS_DIFF_EQS}.
Consequently, we introduce an algorithm that computes a minimal \emph{differential primary decomposition} for arbitrary finitely generated modules (see \autoref{subsect_diff_prim_dec} for the precise definitions). 
This algorithm, along with others (see \autoref{sect_diff_algos}), have been implemented in the computer algebra system \texttt{Macaulay2} \cite{MACAULAY2}.

\medskip

The program of characterizing ideal membership in a polynomial ring with differential conditions was initiated by Gr\"obner \cite{Groebner} in the 1930s, and he successfully employed Macaulay's theory of inverse systems to characterize membership in an ideal primary to a rational maximal ideal. 
Nevertheless, a complete differential characterization of primary submodules over a polynomial ring was obtained in 1970 by analysts, in the form of the \emph{Fundamental Principle} of Ehrenpreis \cite{Ehrenpreis} and Palamodov \cite{PALAMODOV}.
Subsequent algebraic approaches were given in \cite{BRUMFIEL_DIFF_PRIM} and \cite{OBERST_NOETH_OPS}.
More recently, the study of primary ideals and primary submodules via differential operators has continued in e.g. \cite{DAMIANO}, \cite{NOETH_OPS}, \cite{PRIM_IDEALS_DIFF_EQS}, \cite{CHKL}, \cite{CCHKL} and \cite{DIFF_PRIM_DEC}.

Let $D_n$ denote the \emph{Weyl algebra} $D_n = \DiffR(R, R) = R \langle \partial_{x_1},\ldots, \partial_{x_n} \rangle$.
Let $\pp \in \Spec(R)$ be a prime ideal, and $U \subseteq R^r$ a $\pp$-primary $R$-submodule of a free $R$-module of rank $r$.
Following Palamodov's terminology, we say that $\delta_1,\ldots,\delta_m \in (D_n)^r \cong \DiffR(R^r,R)$ is a set of \emph{Noetherian operators} representing $U$ if we have the equality 
$$
U = \big\lbrace w \in R^r \mid \delta_i(w) \in \pp \text{ for all } 1 \le i \le m \big\rbrace.
$$
In a similar fashion to \cite{PRIM_IDEALS_DIFF_EQS}, we parametrize primary submodules via a number of different sets of objects, one of which yields a set of Noetherian operators (see \autoref{thm_main}). 
We provide an algorithm that computes a set of Noetherian operators for a submodule in \autoref{algo_noeth_ops}.
In the other direction, we give an algorithm that computes the submodule corresponding to a set of Noetherian operators in \autoref{algo_backwards}.
The following example displays some of the gadgets used in \autoref{thm_main}.

\begin{example}
	Let $R = \mathbb{Q}[x_1,x_2,x_3,x_4]$ and $\pp = \left(x_1-x_3, x_2-x_4\right) \in \Spec(R)$.
	The $R$-submodule
	$$
	U = \text{image}_R
	\begin{small}
		\begin{bmatrix}
		{x}_{1}-{x}_{3}&0&{x}_{2}-{x}_{4}&0\\
		-{x}_{2}+{x}_{4}&{x}_{1}-{x}_{3}&{x}_{2}{x}_{3}-{x}_{3}{x}_{4}&{x}_{2}^{2}-2\,{x}_{2}{x}_{4}+{x}_{4}^{2}
	\end{bmatrix}
	\end{small}
	\subseteq R^2
	$$
	is $\pp$-primary of multiplicity $3$ over $\pp$.
	Let $\FF = R_\pp/\pp R_\pp = \mathbb{Q}(\overline{x_1}, \overline{x_2})$ be the residue field of $\pp$, where $\overline{x_i} \in \FF$ denotes the class of $x_i \in R$.
	Under the bijective correspondence \hyperref[part_a]{(a)} $\leftrightarrow$ \hyperref[part_b]{(b)} of \autoref{thm_main}, we obtain a $\FF[[y_1,y_2]]$-submodule corresponding to $U$, namely
	$$
	\VV = \text{image}_{\FF[[y_1,y_2]]}
	\begin{small}
	\begin{bmatrix}
		y_1   & 0     & y_2\\
		-y_2 & y_1 & \overline{x_1}y_2
	\end{bmatrix}
	\end{small}
	\subseteq \FF[[y_1,y_2]]^2.
	$$
	Since $\dim_\FF\big(\FF[[y_1,y_2]]^2/\VV\big) = 3$, the submodule $\VV$ corresponds to a point in the punctual Quot scheme $\Quot^3\big(\FF[[y_1,y_2]]^2\big)$.
	Employing the correspondences \hyperref[part_b]{(b)} $\leftrightarrow$ \hyperref[part_c]{(c)} and \hyperref[part_c]{(c)} $\leftrightarrow$ \hyperref[part_d]{(d)} of \autoref{thm_main}, we get the following set of Noetherian operators $\delta_1 = \begin{small} 
		\begin{bmatrix} 
			1  \\
			0 
		\end{bmatrix}
	\end{small}, \;
	\delta_2 = \begin{small} 
		\begin{bmatrix} 
			0 \\
			1 
		\end{bmatrix}
	\end{small}, \;
	\delta_3 = \begin{small} 
		\begin{bmatrix} 
			\partial_{x_1}-x_1\partial_{x_2}  \\
			 \partial_{x_2}  
		\end{bmatrix} 
	\end{small} \in (D_4)^2
	$
	for $U$. In other words, the following equality holds
	$$
	U =  \left\lbrace (w_1,w_2) \in R^2 \, \mid \,  w_1 \in \pp, \; w_2 \in \pp  \,\text{ and }\, \frac{\partial w_1}{\partial_{x_1}} - x_1 \frac{\partial w_1}{\partial_{x_2}} + \frac{\partial w_2}{\partial_{x_2}} \in \pp \right\rbrace.
	$$
	Therefore, instead of describing $U$ via its generators, one could do so with the Noetherian operators $\delta_1,\delta_2,\delta_3$, or with the point in $\Quot^3\big(\FF[[y_1,y_2]]^2\big)$ given by $\VV$.
\end{example}

Recently, the notion of a \emph{differential primary decomposition} for a module was introduced in \cite{DIFF_PRIM_DEC}.
This notion is a natural generalization of Noetherian operators for (not necessarily primary) modules.
Let $U \subseteq R^r$ be an $R$-submodule with associated primes $\Ass_R(R^r/U) = \{ \pp_1,\ldots, \pp_k \} \subseteq \Spec(R)$.
We now wish to describe $U$ in the following way 
\begin{equation}
	\label{eq_dumb}
	U =  \big\lbrace w \in R^r \mid \delta(w) \in \pp_i \text{ for  all }  \delta \in \AAA_i
	\text{ and } 1 \le i \le k \big\rbrace, 
\end{equation}
where each $\AAA_i \subseteq (D_n)^r$ is a finite set of differential operators (for more details, see \autoref{subsect_diff_prim_dec}).
In \cite{DIFF_PRIM_DEC}, it was shown that there exists a differential primary decomposition for $U$ of size equal to the \emph{arithmetic multiplicity} of $U$ (see \autoref{def:diffPrimDec}, \ref{def:amult}) and that this is the minimal possible size.
We note that the existence of differential primary decompositions can always be achieved from the existence of Noetherian operators for a primary submodule. 
Indeed, if $U  \subset R^r$ is an $R$-submodule with primary decomposition $U = \bigcap_{i=1}^k U_i$, after choosing a set of Noetherian operators $\AAA_i$ for $U_i$ (which can be obtained from \autoref{thm_main}), one obtains a differential primary decomposition for $U$, and in particular,  equation \autoref{eq_dumb} holds.
Therefore, the challenge is to compute a \emph{minimal} differential primary decomposition, which yields a measure of complexity from a differential point of view.
Typically a minimal differential primary decomposition does not directly give an ordinary primary decomposition (see \autoref{examp_dumb} and the preceding discussion); however, see \autoref{rem_dumb}.

Building on \autoref{thm_main} and results from \cite{DIFF_PRIM_DEC}, we introduce an algorithm for computing minimal differential primary decompositions (see \autoref{algo_diff_prim_dec}).
This algorithm is an extension of \cite[Algorithm 5.4]{DIFF_PRIM_DEC} from ideals to modules.
The following example shows that a minimal differential primary decomposition need not be obtained by concatenating sets of Noetherian operators for each primary component.

\begin{example}
	Let $R = \mathbb{Q}[x_1,x_2,x_3]$ and consider the $R$-submodule $U = \text{image}_R 
	\begin{small}
	\begin{bmatrix}
		{x}_{1}^{2}&{x}_{1}{x}_{2}&{x}_{2}^{2}\\
		{x}_{2}^{2}&{x}_{2}{x}_{3}&{x}_{3}^{2}
	\end{bmatrix}
	\end{small}
	\subseteq R^2
	$.
	Following the algorithm of \autoref{sect_symb_prim_dec}, we can compute a primary decomposition $U = U_1 \cap U_2 \cap U_3$, where 
	\begin{align*}
	U_1 = \text{image}_R 
	\begin{small}
		\begin{bmatrix}
			{x}_{1}&{x}_{2}^{2}&0\\
			{x}_{3}&{x}_{3}^{2}&{x}_{2}^{2}-{x}_{1}{x}_{3}
		\end{bmatrix}
	\end{small},& \quad 
	U_2 = \text{image}_R 
	\begin{small}
	\begin{bmatrix}
		{x}_{3}&{x}_{2}^{2}&0&{x}_{1}{x}_{2}&{x}_{1}^{2}\\
		0&0&{x}_{3}^{2}&{x}_{2}{x}_{3}&{x}_{2}^{2}
	\end{bmatrix},
	\end{small}\\
	U_3 = \text{image}_R &
	\begin{small}
		\begin{bmatrix}
			0&{x}_{1}^{2}&{x}_{2}^{2}&{x}_{1}{x}_{2}&0\\
			{x}_{1}&0&{x}_{3}^{2}&{x}_{2}{x}_{3}&{x}_{2}^{2}
		\end{bmatrix}
	\end{small}.
	\end{align*}
	The $R$-submodules $U_i$ are primary with associated primes
	$\pp_1 = \left(x_2^2-x_1x_3\right)$,
	$\pp_2 = \left(x_2,x_3\right)$,
	$\pp_3 = \left(x_1,x_2\right)$, respectively.
	The multiplicities of $\{ U_1, U_2, U_3 \}$ over $\{ \pp_1, \pp_2, \pp_3 \}$ are $\{ 1, 5, 5 \}$.
	By \autoref{thm_main}, we can describe $U_1$, $U_2$, $U_3$ by sets of Noetherian operators of sizes $1$, $5$, $5$, respectively. 
	These sets of Noetherian operators give a differential primary decomposition for $U$ of size $11$.
	However, this naive computation is not optimal, as $\amult(U) = 3$.
	Indeed, a minimal differential primary decomposition for $U$ is given by
	$$
	U = \left\lbrace (w_1,w_2) \in R^2 \, \mid \,  -x_3w_1 + x_1w_2 \in \pp_1, \frac{\partial w_2}{\partial_{x_3}} \in \pp_2, \frac{\partial w_1}{\partial_{x_1}} \in \pp_3  \right\rbrace,
	$$
	with only one differential operator needed per associated prime.
\end{example}


\medskip

The basic outline of this paper is as follows.
In \autoref{sect_symb_prim_dec}, we review classical primary decomposition, and present a general algorithm for modules.
In \autoref{section_main_thm}, we prove our representation theorem (\autoref{thm_main}) for primary submodules of a free module, as well as an extension to arbitrary finitely generated modules (\autoref{cor_main}).
In \autoref{sect_diff_algos}, we present several algorithms of a differential nature, including one for computing minimal differential primary decompositions.
In \autoref{sect_joins}, we present an intrinsic differential description of certain ideals that come from the join construction.
Finally, in \autoref{sect_computations}, we illustrate the various algorithms on examples with our \texttt{Macaulay2} implementation.

\section{A general primary decomposition algorithm} \label{sect_symb_prim_dec}

We begin with a general algorithm for primary decomposition of modules, inspired by the work of Eisenbud-Huneke-Vasconcelos \cite{EHV} for ideals, with a particular focus on computational aspects.
Although primary decomposition for ideals has been well-studied in the literature, the case of modules is considerably less prominent (which we hope to remedy with this exposition!); cf. \cite{rutman1992grobner} as well as \cite{Indrees, idrees2014primary, idrees2015algorithm} for some treatments.

We start with some reductions to the main case of interest.
First, we reduce to the case of polynomial rings: let $\kk$ be an arbitrary field, and $T$ a finitely generated $\kk$-algebra.
If $R = \kk[x_1, \ldots, x_n]$ is a polynomial ring with a surjection $\pi : R \twoheadrightarrow T$,
then for any $T$-module $N$, one has $\Ass_R(\null_R N) = \{ \pi^{-1}(P) \mid P \in \Ass_T(N) \}$, where $\null_R N$ denotes $N$ viewed as an $R$-module.
In this way we may compute associated primes and primary components of $N$ over $T$, by first computing that of $\null_R N$ over $R$, and then applying $\pi$.

Next, to simplify the notation on modules, note that if $M' \subseteq M$ is a submodule, then a primary decomposition of $M'$ in $M$ can be obtained by lifting a primary decomposition of $0$ in $M/M'$.
Thus we may always take $M' = 0$ (a benefit afforded by working with general modules), and state the problem as follows: for a finitely generated module $M \ne 0$ over a polynomial ring $R$, find submodules $Q_1, \ldots, Q_s \subseteq M$ such that $\bigcap_{i=1}^s Q_i = 0$ and $|\Ass_R(M/Q_i)| = 1$.
The decomposition should moreover be \emph{minimal}, in the sense that $\bigcap_{j \ne i} Q_j \ne 0$ for all $i$, and also $\Ass_R(M/Q_i) = \Ass_R(M/Q_j) \iff i = j$.

The primary decomposition algorithm described here proceeds in 2 steps: first, find all associated primes of $M$, and second, determine valid $P_i$-primary components $Q_i$ for each associated prime $P_i$ (note that by uniqueness of associated primes from a primary decomposition, such a decomposition will automatically be minimal).
For the first step, following \cite{EHV}, we first reduce the problem of computing all associated primes of a module, to computing minimal primes of ideals:

\begin{theorem}[{\cite[Theorem 1.1]{EHV}}] \label{thm:associatedPrimesEHV}
	For any $i \ge 0$, the associated primes of $M$ of codimension $i$ are precisely the minimal primes of $\ann \Ext_R^i(M, R)$ of codimension $i$.
\end{theorem}

In view of this, we may compute $\Ass_R(M)$ via oracles to compute (1) a free resolution of a module $M$ (and thus any Ext modules $\Ext_R^\bullet(M, \cdot)$), and (2) minimal primes of any ideal $I \subseteq R$, which we henceforth assume are given (in practice, both are well-optimized in \texttt{Macaulay2}). 
Note that following the above procedure iteratively will naturally produce a list of associated primes which are weakly ordered by codimension (e.g. all codimension 1 primes appear before any codimension 2 primes, etc.).

For the second step, namely producing valid primary components, we proceed inductively.
Order the associated primes $P_1, \ldots, P_s$ of $M$ by a linear extension of the partial order by inclusion, i.e. $P_i \subseteq P_j \implies i \le j$ (note that this is automatic if the associated primes are weakly ordered by codimension, as in the previous paragraph).
In particular, as a base case $P_1$ is a minimal prime of $M$ (i.e. a minimal prime of $\ann M$).
Primary components to minimal primes are uniquely determined, and can be obtained as follows:

\begin{proposition} \label{prop:localPrimaryComponents}
	Let $P \in \Spec(R)$, and let $M \to M_P$ be the localization map. Then:
	
	\begin{enumerate}
		
		\item \cite[Theorem 3.10(d)]{EISEN_COMM} $\ker(M \to M_P)$ equals the intersection of all $P_i$-primary components of $0$ in $M$ for $P_i \in \Ass_R(M)$, $P_i \subseteq P$ (in particular, this intersection is uniquely determined by $M$ and $P$).
		
		\item \cite[Proposition 3.13]{EISEN_COMM} Suppose $f \in R$ is such that for all $P_i \in \Ass_R(M)$, one has $f \in P_i \iff P_i \not \subseteq P$. 
		Then $\ker(M \to M_P) = \ker(M \to M_f) = 0 :_M f^\infty$.
		
	\end{enumerate}
\end{proposition}

An element $f$ as in \autoref{prop:localPrimaryComponents}(2) can be obtained as follows: for each associated prime $P_i$ not contained in $P$, choose a generator $g_i$ of $P_i$ not contained in $P$; then take $f := \prod_{P_i \not \subseteq P} g_i$.
Taking $P = P_i$ for some minimal prime $P_i$ of $M$ in \autoref{prop:localPrimaryComponents} shows that given an oracle to compute saturations, we may obtain (the unique) primary components corresponding to minimal primes of $M$.

It then remains to compute a valid $P$-primary component $Q$, for an embedded prime $P$.
In this case such a $Q$ is not unique; indeed there are always infinitely many valid choices for $Q$.
The next proposition gives one class of such choices:

\begin{proposition}[\cite{EHV}, p. 27-28] \label{prop:embCompCandidate}
	For $P \in \Ass_R(M)$ and $j \ge 0$, fix generators $P = (f_1, \ldots, f_m)$ for $P$, and set $P^{[j]} := (f_1^j, \ldots, f_m^j)$. 
	Then for $j \gg 0$, the submodule $Q_{[j]} := \hull(P^{[j]}M, M)$ is a valid $P$-primary component of $0$ in $M$.
\end{proposition}

Here $\hull(N, M)$ is the equidimensional hull of $N$ in $M$, i.e. the intersection of all primary components of $N$ of maximal dimension.
There are a number of ways to compute $Q_{[j]}$: the first is via \autoref{prop:localPrimaryComponents}, viz. $Q_{[j]} = \ker(M \to (M/P^{[j]}M)_P)$.
Another method is given in \cite[Theorem 1.1(2)]{EHV}, which is a general way to compute hulls via iterated Ext modules, and yet another method is given in \cite[Algorithm 1.2]{EHV}.

To find a stopping criterion for the exponent $j$ in \autoref{prop:embCompCandidate}, note that inductively, we may assume that $P_i$-primary components for any $P_i \subsetneq P$ have already been computed, and so their intersection $\mathcal{V} := \bigcap_{P_i \subsetneq P} Q_i$ is also known.
By \autoref{prop:localPrimaryComponents}, we also know $\mathcal{U} := \bigcap_{P_i \subseteq P} Q_i$.
Then for any $j \ge 0$, the submodule $Q_{[j]}$ in \autoref{prop:embCompCandidate} is a valid $P$-primary component if and only if $Q_{[j]} \cap \mathcal{V} = \mathcal{U}$.
We may thus find a valid $P$-primary component as follows: initialize $j$ at some starting value, and compute $Q_{[j]}$.
If $Q_{[j]} \cap \mathcal{V} \ne \mathcal{U}$, then increment $j$, and repeat until a valid candidate $Q_{[j]}$ is found.

\begin{remark}
	A few remarks are in order concerning efficiency of the algorithm described above:
	\begin{enumerate}
		\item Both choices of starting value of $j$, and the function used to increment $j$, are relevant considerations for efficiency of the algorithm.
		If the starting value of $j$ is too small, or the increment function grows too slowly, then invalid candidates may be computed many times.
		On the other hand, the computation time for $Q_{[j]}$ tends to increase as $j$ increases, so it is desirable not to take $j$ unnecessarily large.
		The current implementation in \texttt{Macaulay2} uses the increment $j \to \lceil \frac{3j}{2} \rceil$ (the starting value is more complicated, depending on the degrees of generators of $P$ and of $\ann M$).
		
		\item The use of bracket powers $P^{[j]}$ rather than ordinary powers $P^j$ in \autoref{prop:embCompCandidate} is also for efficiency: if either $j$ or the number of generators of $P$ is large, then $P^j$ may take much longer to compute than $P^{[j]}$.
		
		\item Of the three methods given above for computing embedded components $Q_{[j]}$, usually the first method (namely as a kernel of a localization map) is the most efficient, although this is not always the case (in some examples, the second method can be drastically faster).
		Note that in the first method, it is necessary to compute $\Ass_R(M/P^{[j]}M)$, which is in general strictly bigger than $\{ P \}$.
		
		\item The necessity of realizing $\ker(M \to M_P)$ as a saturation in \autoref{prop:localPrimaryComponents}(2) stems from the fact that $M_P$ typically does not have a finite presentation as an $R$-module. 
		Although one could also express $\ker(M \to M_P)$ as a saturation $0 :_M (P')^\infty$, where $P'$ equals the intersection of all associated primes of $M$ not contained in $P$, it is almost always much more efficient to compute saturations by a single element, than by general ideals.
		
		\item As a general rule, computation of associated primes is the most time-consuming step in this procedure.
		Once the associated primes are known, the minimal primary components tend to be computed very quickly, and the time for computing the embedded components can vary based on the method chosen (cf. point (3) above).
		In particular, this algorithm tends to perform well for modules whose free resolutions can be cheaply computed (e.g. when the number of variables is small).
		
	\end{enumerate}
\end{remark}

The main differences between the algorithm presented here and the one given in \cite[Section 4]{EHV} are the stopping criterion for the exponent $j$ in \autoref{prop:embCompCandidate}, the use of bracket powers as opposed to ordinary powers, and the methods used to compute hulls.
In summary, the algorithm described here reduces general primary decomposition for modules to the following tasks (some of which can be seen as special cases of primary decomposition):

\begin{enumerate}
	\item computation of Ext modules (in fact, $\Ext_R^\bullet(\cdot, R)$ suffices),
	
	\item computation of minimal primes of an ideal $I \subseteq R$,
	
	\item computation of colon modules (of which saturations and annihilators are special cases),
	
	\item computation of intersections of submodules.
\end{enumerate}

\section{A representation theorem for primary submodules}
\label{section_main_thm}

We now leave the classical picture, and adopt a differential point of view.
Our main theorem in this section parametrizes primary submodules of a free module of finite rank in terms of \emph{punctual Quot schemes}, \emph{vector spaces closed under differentiation} and \emph{subbimodules of the Weyl-Noether module}.
This extends the main result of \cite{PRIM_IDEALS_DIFF_EQS} to the case of modules. 
As a simple corollary, we also obtain a representation theorem for primary submodules of an arbitrary module.

\begin{setup} \label{setup_1}
For the rest of this section we fix the following notation:
\begin{itemize}[--]
\item Let $\kk$ be a field of characteristic zero, and $R := \kk[x_1,\ldots,x_n]$ a polynomial ring over $\kk$.
\item For an integer $r \ge 0$, let $R^r$ be a free $R$-module of rank $r$.
\item Let $\pp \in \Spec(R)$ be a prime ideal with codimension $c := \HT(\pp)$.
\item The residue field of $\pp$ is denoted $\FF := k(\pp) = \Quot(R/\pp) = R_\pp/\pp R_\pp$.
\item A subset of variables $\{ x_{i_1}, \ldots, x_{i_\ell} \} \subseteq \{x_1,\ldots,x_n\}$ is \emph{independent modulo $\pp$} if their images in $R/\pp$ are algebraically independent over $\kk$, or equivalently $\kk[x_{i_1}, \ldots, x_{i_\ell}] \cap \pp = \{0\}$.
After possibly permuting the variables, we may assume that $\{x_{c+1},\ldots,x_n\}$ is a \emph{basis modulo $\pp$}, i.e. a maximal set of independent variables modulo $\pp$ (see \cite[Example 13.2]{Mateusz}).
\item Let $\LL := \kk(x_{c+1},\ldots,x_n)$ denote the field of rational functions in the basis variables (which is a purely transcendental extension of $\kk$), and $S$ be the polynomial ring 
$
S := \kk(x_{c+1},\ldots,x_n)[x_1,\ldots,x_c]$ (which is a localization of $R$, as $S \cong \LL \otimes_{\kk[x_{c+1},\ldots,x_n]} R$).
\item The \emph{Weyl algebra} and the \emph{relative Weyl algebra} are denoted by $D_n := R \big< \partial_{x_1},\ldots,\partial_{x_n} \big>$ and $D_{n,c} := R \big< \partial_{x_1},\ldots,\partial_{x_c} \big> \subseteq D_n$, respectively.
\item The \emph{multiplicity} of a $\pp$-primary submodule $U \subseteq R^r$ is defined as $\leng_{R_\pp}\left(R_\pp^r/U_\pp\right)$.
\item For an integer $m \ge 0$, the \emph{punctual Quot scheme} is a parameter space ${\rm Quot}^m\left(\FF[[y_1,\ldots,y_c]]^r\right)$ whose $\FF$-points parametrize all $\FF[[y_1,\ldots,y_c]]$-submodules $\VV \subseteq \FF[[y_1,\ldots,y_c]]^r$ of colength $m$, i.e. which satisfy
$\dim_\FF\left(\FF[[y_1,\ldots,y_c]]^r/\VV\right) = m$.
\item We say that $\delta_1,\ldots,\delta_m \in (D_n)^r \cong \DiffR(R^r,R)$ is a set of \emph{Noetherian operators} for a $\pp$-primary submodule $U \subseteq R^r$ if the following equality holds
\begin{equation} \label{eq_noeth_ops}
U = \big\lbrace w \in R^r \mid \delta_i(w) \in \pp \text{ for all } 1 \le i \le m \big\rbrace.
\end{equation}
\end{itemize}
\end{setup}

We can now state our main result:

\begin{theorem} 
	\label{thm_main}
	The following four sets of objects are in bijective correspondence:
	\begin{enumerate}[(a)]
		\item \label{part_a} $\pp$-primary $R$-submodules $U \subseteq R^r$ of multiplicity $m$ over $\pp$, 
		\item \label{part_b} $\FF$-points in the punctual Quot scheme ${\rm Quot}^m\left(\FF[[y_1,\ldots,y_c]]^r\right)$,
		\item \label{part_c} $m$-dimensional $\FF$-subspaces of $\,\FF[z_1,\ldots,z_c]^r$ that are closed under differentiation,
		\item \label{part_d}  $m$-dimensional $\FF$-subspaces of
		the Weyl-Noether module $\FF \otimes_R  (D_{n,c})^r$ that are $R$-bimodules. 
	\end{enumerate}
	Moreover, any basis of the $\FF$-subspace in part \hyperref[part_d]{(d)}  can be lifted to Noetherian operators $\delta_1,\ldots,\delta_m \in (D_{n,c})^r$ for the $R$-submodule $U$ in part \hyperref[part_a]{(a)}.
\end{theorem}

We structure the proof of \autoref{thm_main} as follows.
The correspondence \hyperref[part_a]{(a)} $\leftrightarrow$ \hyperref[part_b]{(b)} is detailed in \autoref{thm:param_primary}.
The map $\gamma$ defined in \autoref{eq_map_gamma} yields a bijection
\begin{equation} \label{eq_corr_a_b}
\begin{array}{ccc}
	\left\lbrace\begin{array}{c}
		\mbox{$\pp$-primary $R$-submodules of $R^r$}\\
		\mbox{of multiplicity $m$ over $\pp$}
	\end{array}\right\rbrace
	
	& \longleftrightarrow &
	
	\left\lbrace\begin{array}{c}
		\mbox{points in }{\rm Quot}^m(\FF[[y_1,\ldots,y_c]]^r)\\
	\end{array}\right\rbrace\\
	
	U & \longrightarrow & \VV = \FF[[y_1,\ldots,y_c]] \cdot \gamma(U) + (y_1,\ldots,y_c)^m\FF[[y_1,\ldots,y_c]]^r\\
	U=\gamma^{-1}(\VV) & \longleftarrow & \VV.
\end{array}
\end{equation}

The correspondence \hyperref[part_b]{(b)} $\leftrightarrow$ \hyperref[part_c]{(c)} is detailed in \autoref{thm:Macaulay_dual}.
We regard the polynomial ring $\FF[z_1,\ldots,z_c]$ as an $\FF[[y_1,\ldots,y_c]]$-module by letting $y_i$ act as $\partial_{z_i}$, i.e. $y_i \cdot F := \frac{\partial F}{\partial_{z_i}}$ for any $F \in \FF[z_1,\ldots,z_c]$.
By Macaulay inverse systems (see also \autoref{eq_V_perp}, \autoref{eq_W_perp}) we have a bijection
\begin{equation}
	\label{eq_corr_b_c}
\begin{array}{ccc}
	\left\lbrace\begin{array}{c}
		\mbox{points in }{\rm Quot}^m\left(\FF[[y_1,\ldots,y_c]]^r\right)\\
	\end{array}\right\rbrace
	
	& \longleftrightarrow &
	
	\left\lbrace\begin{array}{c}
		\mbox{$m$-dimensional $\FF$-subspaces of}\\
		\mbox{$\FF[z_1,\ldots,z_c]^r$ closed under differentiation}
	\end{array}\right\rbrace\\

	\VV & \longrightarrow & \WW=\VV^\perp\\
	\VV= \WW^\perp & \longleftarrow & \WW.\\
\end{array}
\end{equation}

Finally, the correspondence \hyperref[part_c]{(c)} $\leftrightarrow$ \hyperref[part_d]{(d)} is detailed in \autoref{subsect_proof_represent}.
The map $\Omega$ defined in \autoref{eq_map_omega} yields a bijection 
\begin{equation}
	\label{eq_corr_c_d}
\begin{array}{ccc}
	\left\lbrace\begin{array}{c}
		\mbox{$m$-dimensional $\FF$-subspaces of}\\
	\mbox{$\FF[z_1,\ldots,z_c]^r$ closed under differentiation}
	\end{array}\right\rbrace
	
	& \longleftrightarrow &
	
	\left\lbrace\begin{array}{c}
		\mbox{$m$-dimensional $\FF$-subspaces of}\\
		\mbox{$\FF \otimes_R  (D_{n,c})^r$ that are $R$-bimodules. }
	\end{array}\right\rbrace\\

	\WW & \longrightarrow & \mathcal{E}= \Omega(W)\\
	\WW= \Omega^{-1}(\mathcal{E}) & \longleftarrow & \mathcal{E}.\\
\end{array}
\end{equation}
Furthermore, by \autoref{lem_descript_diff_op}, we can lift elements from $\FF \otimes_R D_{n,c}$ to $D_{n,c}$.

\medskip

We now extend \autoref{thm_main} to arbitrary modules.
Let $M$ be a finitely generated $R$-module which can be generated by $r$ elements, so that there is a short exact sequence of $R$-modules
\begin{equation} \label{eq_present_M}
	0 \to K \to R^r \to M \to 0.
\end{equation}
Let $U \subseteq M$ be a $\pp$-primary $R$-submodule of $M$ of multiplicity $m = \text{length}_{R_\pp}(M_\pp/U_\pp)$ over $\pp$.
There is a unique $R$-submodule $\widetilde{U} \subseteq R^r$ containing $K$ such that $\widetilde{U} / K \cong U$, which we call the \emph{lift} of $U$ to $R^r$.
Since $R^r/\widetilde{U} \cong M/U$, it follows that $\widetilde{U}$ is a $\pp$-primary submodule of $R^r$ with the same multiplicity as $U$ over $\pp$.
This convenient fact allows us to lift $\pp$-primary submodules of $M$ to $\pp$-primary submodules of $R^r$.
In terms of the syzygies $K \subseteq R^r$ of $M$ in \autoref{eq_present_M}, we define the following objects:
\begin{itemize}[--]
	\item Let $\VV' \subseteq \FF[[y_1,\ldots,y_c]]^r$ be the $\FF[[y_1,\ldots,y_c]]$-submodule 
	$$
	\VV' := \FF[[y_1,\ldots,y_c]] \cdot \gamma(K) + (y_1,\ldots,y_c)^m\FF[[y_1,\ldots,y_c]]^r \subseteq \FF[[y_1,\ldots,y_c]]^r.
	$$
	\item Let $\WW' := (\VV')^\perp$ be the corresponding $m$-dimensional $\FF$-subspace of $\FF[z_1,\ldots,z_c]^r$ closed under differentiation.
	\item Let $\mathcal{E}' := \Omega(W')$ be the resulting $m$-dimensional $\FF$-subspace of $\FF \otimes_R  (D_{n,c})^r$ which is an $R$-bimodule.
\end{itemize}

We can now state the extension of \autoref{thm_main} to an arbitrary finitely generated $R$-module.

\begin{corollary}
	\label{cor_main}
	With the above notation, the following four sets of objects are in bijective correspondence:
	\begin{enumerate}[(a)]
		\item \label{part_aa} $\pp$-primary $R$-submodules $U \subseteq M$ of multiplicity $m$ over $\pp$, 
		\item \label{part_bb} $\FF$-points $\VV \subseteq \FF[[y_1,\ldots,y_c]]^r$ in the punctual Quot scheme ${\rm Quot}^m\left(\FF[[y_1,\ldots,y_c]]^r\right)$ with $\VV \supset \VV'$,
		\item \label{part_cc} $m$-dimensional $\FF$-subspaces $\WW \subseteq \FF[z_1,\ldots,z_c]^r$ that are closed under differentiation with $\WW \subseteq \WW'$,
		\item \label{part_dd}  $m$-dimensional $\FF$-subspaces $\mathcal{E} \subseteq \FF \otimes_R  (D_{n,c})^r$ of
		the Weyl-Noether module that are $R$-bimodules with $\mathcal{E} \subseteq \mathcal{E}'$.  
	\end{enumerate}
	Moreover, any basis of the $\FF$-subspace in part \hyperref[part_dd]{(d)} can be lifted to Noetherian operators $\delta_1,\ldots,\delta_m \in (D_{n,c})^r$ for the lift $\widetilde{U} \subseteq R^r$ of the $R$-submodule $U$ in part \hyperref[part_aa]{(a)}.
\end{corollary}

\subsection{A basic recap on differential operators}
In this subsection, we recall basic properties of differential operators to be used in the proof of the main theorem (for further details, the reader is referred to \cite[\S 16]{EGAIV_IV}).
For $R$-modules $M$ and $N$, we regard $\Hom_\kk(M, N)$ as an $(R\otimes_\kk R)$-module, by setting 
$$
\left((s \otimes_\kk t) \delta \right)(w) = s \delta(tw) \quad \text{ for all } \delta \in \Hom_\kk(M, N), \; w \in M,\; s,t \in R. 
$$ 
We use the bracket notation $[\delta,s](w) := \delta(sw)-s\delta(w)$ for $\delta \in \Hom_\kk(M, N)$, $s \in R$ and $w \in M$.

Unless otherwise specified, whenever we consider an $(R \otimes_\kk R)$-module as an $R$-module, we do so by letting $R$ act via the left factor of $R \otimes_\kk R$. 

\begin{definition}
	\label{def_diff_ops}
	Let $M$ and $N$ be $R$-modules.
	The \textit{$m$-th order $\kk$-linear differential operators}, denoted
	$\Diff_{R/\kk}^m(M, N) \subseteq \Hom_\kk(M, N)$,
	form an $(R\otimes_\kk R)$-module that is defined inductively~by
	\begin{enumerate}[\rm (i)]
		\item $\Diff_{R/\kk}^{0}(M,N) := \Hom_R(M,N)$.
		\item $\Diff_{R/\kk}^{m}(M, N) := 
		\big\lbrace \delta \in \Hom_\kk(M,N) \,\mid\, [\delta, s] \in \Diff_{R/\kk}^{m-1}(M, N) 
		\text{ for all } s \in R \big\rbrace$.
	\end{enumerate}
	The set of all \textit{$\kk$-linear differential operators from $M$ to $N$} 
	is the $(R \otimes_\kk R)$-module
	$$
	\Diff_{R/\kk}(M, N) := \bigcup_{m=0}^\infty \Diff_{R/\kk}^m(M,N).
	$$
	Following the notation used in \cite{NOETH_OPS, PRIM_IDEALS_DIFF_EQS, DIFF_PRIM_DEC}, subsets $\mathcal{E} \subseteq \Diff_{R/\kk}(M, N)$ are viewed as systems of differential equations, and their solution spaces over $\kk$ are defined as
	\begin{equation*}
		{\rm Sol}(\mathcal{E}) := \big\lbrace w \in M \,\mid \, \delta(w) = 0 \text{ for all } \delta \in \mathcal{E} \big\rbrace 
		= \bigcap_{\delta \in \mathcal{E} } {\rm Ker}(\delta).
	\end{equation*}
\end{definition}

\begin{example}
As $R = \kk[x_1,\ldots,x_n]$ is a polynomial ring over a field $\kk$ of characteristic zero, $\Diff_{R/\kk}(R,R)$ is the Weyl algebra $D_n = R\langle \partial_{x_1},\ldots, \partial_{x_n} \rangle = \bigoplus_{\alpha \in \NN^n} R \partial_\xx^\alpha$.
\end{example}

We next describe differential operators via the module of principal parts.
Consider the multiplication map $\mu : R \otimes_\kk R \to R$, $s \otimes_\kk t \mapsto st$, and define $\Delta_{R/\kk} := \Ker(\mu)$, which is an ideal in $R \otimes_\kk R$.
One can alternatively define differential operators as follows:

\begin{proposition}[{\cite[Proposition 2.2.3]{AFFINE_HOPF_I}}] \label{prop_Diff_ann_diag}
Let $M, N$ be $R$-modules.
Then $\DiffR^m(M, N)$ is the $(R \otimes_\kk R)$-submodule of $\Hom_\kk(M, N)$ annihilated by $\Diag^{m+1}$.
\end{proposition}

\begin{definition}
	Let $M$ be an $R$-module.
	The {\em module of $m$-th principal parts} of $M$ is defined as	
	$$
	P_{R/\kk}^m(M) := \frac{R \otimes_\kk M}{	\Delta_{R/\kk}^{m+1}  \left(R \otimes_\kk M\right)}.
	$$
	This is a module over $R \otimes_\kk R$ and thus also over $R$.
	For simplicity, set $ P_{R/\kk}^m := P_{R/\kk}^m(R)$.
\end{definition}
For any $R$-module $M$, consider the universal map 
$d^m : M \to P_{R/\kk}^m(M)$, $w  \mapsto \overline{1 \otimes_\kk w} $.
The next result is a fundamental characterization of differential operators.

\begin{proposition}[{\cite[Proposition 16.8.4]{EGAIV_IV}, \cite[Theorem 2.2.6]{AFFINE_HOPF_I}}]
	\label{prop_represen_diff_opp}
	Let $M$ and $N$ be $R$-modules and let $m\ge 0$.
	Then the following map is an isomorphism of $R$-modules:
	\begin{align*}
		{\left(d^m\right)}^* : \Hom_R\left(P_{R/\kk}^m(M), N\right) 
		& \xrightarrow{\cong} \Diff_{R/\kk}^m(M, N), \\
		\varphi  \quad & \mapsto \quad \varphi \circ d^m.
	\end{align*}
\end{proposition}

We recall an explicit description of differential operators on free modules.
Let $J \subseteq R$ be an ideal and consider the canonical map $\pi : R \twoheadrightarrow R/J$.
We wish to describe $\DiffR^m(F,R/J)$ for a free $R$-module $F = R^r$.
Since  $\DiffR^m(F,R/J) \cong {\DiffR^m(R, R/J)}^r$ (see e.g. \cite[Lemma 2.7]{NOETH_OPS}), it is enough to describe $\DiffR^m(R, R/J)$.
We have an induced map:
\begin{equation*}
	\Diff_{R/\KK}^m(\pi) : \Diff_{R/\KK}^m(R, R) \to 
	\Diff_{R/\KK}^m(R, R/J), \quad \delta \mapsto \overline{\delta} = \pi \circ \delta.
\end{equation*}

\begin{lemma}[{\cite[Lemma 2]{PRIM_IDEALS_DIFF_EQS}}]
	\label{lem_descript_diff_op}
	With the above notation, the following statements hold:
		\begin{enumerate}[(i)]
		\item  
		$\Diff_{R/\kk}^m(R,R/J) = \bigoplus_{\lvert \alpha \rvert \le m} (R/J) \overline{\partial_\mathbf{x}^\alpha}$ 
		\, where\,
		$\overline{\partial_\mathbf{x}^\alpha} = \pi \circ \partial_\mathbf{x}^\alpha.$
		\item $\Diff_{R/\kk}^m(\pi)$ is surjective:
		a differential operator $\epsilon = \sum_{\lvert \alpha \rvert \le m} \overline{s_\alpha}
		\overline{\partial_\mathbf{x}^\alpha} \in \Diff_{R/\KK}^m(R,R/J)$ with $s_\alpha \in R$  
		can be lifted to
		$\delta=\sum_{\lvert \alpha \rvert \le m} s_\alpha \partial_\mathbf{x}^\alpha \in \Diff_{R/\kk}^m(R,R)$.
	\end{enumerate}
\end{lemma}

The notation below will be useful for describing the $(R \otimes_\kk R)$-action:

\begin{notation}
	\label{nota_T}
	Let $T := R \otimes_\kk R = \kk[x_1, \ldots, x_n, y_1, \ldots, y_n]$ be a polynomial ring in $2n$ variables, where $x_i$ represents $x_i \otimes_\kk 1$ and $y_i$ represents $1 \otimes_\kk x_i - x_i \otimes_\kk 1$.
	The action of $T$ on $\Hom_\kk(M,N)$ is defined as follows: for all $\delta \in \Hom_\kk(M, N)$ and $w \in M$, 
	$$
	(x_i \cdot \delta) (w) \,:=\, x_i \delta(w) \quad \text{ and } \quad  (y_i \cdot \delta) (w) \,:=\, \delta(x_i w) - x_i\delta(w) = \left[\delta,x_i\right](w)
	$$
	for all $1 \le i \le n$.
\end{notation}

\begin{remark}
	\label{rem_diff_yy}
	Viewing $\Princ^m$ and $\DiffR^m(R,R/J)$ as $T$-modules yields the following useful descriptions:
	\begin{enumerate}[(i)]
		\item $P_{R/\kk}^m \,= \, \bigoplus_{\alpha \in \NN^n, \lvert \alpha \rvert \le m} R y_1^{\alpha_1}\cdots y_n^{\alpha_n}$.
		\item Under the isomorphism $\DiffR^m(R,R/J) \cong \Hom_R(\Princ^m,R/J)$ (cf.~\autoref{prop_represen_diff_opp}), the dual basis element $\left(y_1^{\alpha_1}\cdots y_n^{\alpha_n}\right)^*$ corresponds to the differential operator $\frac{1}{\alpha_1!\cdots \alpha_n!} \overline{\partial_{x_1}^{\alpha_1}\cdots \partial_{x_n}^{\alpha_n}}$.
	\end{enumerate}
\end{remark}
\begin{proof}
	For explicit computations, see \cite[\S 5]{DIFF_PRIM_DEC}.
\end{proof}

The next proposition describes the differential operators of an arbitrary finitely generated module.

\begin{proposition}
	\label{lem_diff_M}
	Let $0 \to K \to F \to M \to 0$ be a short exact sequence of $R$-modules where $F$ is $R$-free of finite rank.
	Let $N$ be an $R$-module.
	Then, for all $m \ge 0$, we have 
	$$
	\Diff_{R/\kk}^m(M, N) \, = \, \big\lbrace \delta \in \Diff_{R/\kk}^m(F, N)  \mid  \delta(K) = 0 \big\rbrace.
	$$
\end{proposition}
\begin{proof}
	By left exactness of $\Diff_{R/\kk}^m(\bullet, N)$, one has $\Diff_{R/\kk}^m(M, N) \hookrightarrow \Diff_{R/\kk}^m(F, N)$ (see e.g. \cite[Lemma 2.7]{NOETH_OPS}).
	The inclusion ``$\subseteq$'' is clear. 
	Conversely, if $\delta \in \Diff_{R/\kk}^m(F,N)$ and $\delta(K) = 0$, then there is a unique map $\overline{\delta} \in \Hom_\kk(M, N)$ induced by $\delta$, and by \autoref{prop_Diff_ann_diag}, $\Diag^{m+1} \cdot \delta = 0$, which implies $\Diag^{m+1} \cdot \overline{\delta} = 0$.
	Then $\overline{\delta} \in \Diff_{R/\kk}^m(M, N)$ by \autoref{prop_Diff_ann_diag} again, and the desired result follows.
\end{proof}

\subsection{Punctual Quot schemes and Macaulay inverse systems}
\label{subsect_Quot_Macaulay}

We are now ready to begin describing the correspondences in our main result.
The purpose of this subsection is to parametrize $\pp$-primary $R$-submodules of a free module $R^r$ via punctual Quot schemes and Macaulay inverse systems.

Let $\BB$ be the power series ring $\BB := \FF[[y_1,\ldots,y_c]]$ over the residue field $\FF= R_\pp/\pp R_\pp$ and denote its maximal ideal by $\nn := (y_1,\ldots,y_c) \subseteq \BB$.
The punctual Quot scheme ${\rm Quot}^m\left(\BB^r\right)$ parametrizes all quotients 
$$
\BB^r \,=\, \FF[[y_1,\ldots,y_c]]^r \,\twoheadrightarrow\, T
$$ 
such that $T$ has finite length $m$, and consequently, $T$ is only supported at the maximal ideal $\nn$.
More precisely, the $\FF$-rational points of ${\rm Quot}^m\left(\BB^r\right)$ correspond to $\BB$-submodules $\VV \subseteq \BB^r$ such that $\dim_\FF(\BB^r/\VV) = m$
(in this paper, \emph{a point in ${\rm Quot}^m\left(\BB^r\right)$} will always mean \emph{an $\FF$-rational point in ${\rm Quot}^m\left(\BB^r\right)$}).
This parameter space has been considered in several papers, e.g. \cite{baranovsky2000moduli}, \cite{ellingsrud1999irreducibility}, \cite{HENNI_GUIMARES}.
When $r = 1$, ${\rm Quot}^m\left(\BB^r\right)$ coincides with the \emph{punctual Hilbert scheme} ${\rm Hilb}^m\left(\FF[[y_1,\ldots,y_c]]\right)$ studied by Briançon \cite{Bri77} and Iarrobino \cite{IARROBINO_HILB}.

The following basic fact shows that any 
point in ${\rm Quot}^m\left(\BB^r\right)$ can be identified with a $(\BB/\nn^m\BB )$-submodule of $\BB^r/\nn^m\BB^r$.

\begin{proposition} \label{rem_contain_power}
For any $\BB$-submodule $\VV \subseteq \BB^r$ with colength $m = \dim_\FF \left( \BB^r/\VV \right)$, one has $\VV \supseteq \nn^m\BB^r$.
\end{proposition}
\begin{proof}
Consider the associated graded module 
\[
\gr(\BB^r/\VV) \;:=\; \bigoplus_{k = 0}^\infty \nn^k\BB^r / \left(\nn^{k+1}\BB^r +  \left( \nn^k\BB^r \cap \VV \right) \right)
\]
which satisfies $\dim_\FF \left( \gr(\BB^r/\VV) \right) = \dim_\FF \left( \BB^r/\VV \right)$.
If $\VV \not\supseteq \nn^m\BB^r$, then $\left[ \gr(\BB^r/\VV) \right]_k \neq 0$ for all $0 \le k \le m$, so $m = \dim_\FF\left(\BB^r/\VV\right) = \dim_\FF\left(\gr(\BB^r/\VV)\right) \ge m+1$, a contradiction.
\end{proof}

Following the approach of \cite{PRIM_IDEALS_DIFF_EQS}, we have an injection
\begin{equation*}
	\eta : R \hookrightarrow \BB\, , \qquad
	\begin{matrix}
		x_i  &\mapsto &  y_i+ \overline{x_i}, & \!\!\!\!\! \mbox{ for }1\leq i\leq c,\\
		x_j & \mapsto  & \overline{x_j},& \quad \,\mbox{ for }c+1\leq j\leq n,
	\end{matrix}
\end{equation*}
where $\overline{x_i}$ denotes the class of $x_i$ in $\FF$ for $1\leq i\leq n$.
We now define the induced injection
\begin{equation}
	\label{eq_map_gamma}
	\gamma : R^r \hookrightarrow \BB^r, \quad (f_1,\ldots,f_r) \in R^r \mapsto (\eta(f_1),\ldots,\eta(f_r)) \in \BB^r.
\end{equation}

The map $\gamma$ provides the first correspondence \hyperref[part_a]{(a)} $\leftrightarrow$ \hyperref[part_b]{(b)} in our main theorem.

\begin{theorem}
	\label{thm:param_primary}
	With the above notation, there is a bijection
	$$
	\begin{array}{ccc}
		\left\lbrace\begin{array}{c}
			\mbox{$\pp$-primary $R$-submodules of $R^r$}\\
			\mbox{of multiplicity $m$ over $\pp$}
		\end{array}\right\rbrace
		
		& \longleftrightarrow &
		
		\left\lbrace\begin{array}{c}
			\mbox{points in }{\rm Quot}^m(\BB^r)\\
		\end{array}\right\rbrace\\
		
		U & \longrightarrow & \VV = \BB \cdot \gamma(U) + \nn^m \BB^r \\
		U=\gamma^{-1}(\VV) & \longleftarrow & \VV.
	\end{array}
	$$
\end{theorem}
\begin{proof}
	The canonical map $R^r \hookrightarrow S^r$, $U \subseteq R^r \mapsto US^r \subseteq S^r$ gives a bijection between $\pp$-primary $R$-submodules and $\pp S$-primary $S$-submodules (see e.g. \cite[Theorem 4.1]{MATSUMURA}).  	
	By \cite[Proposition 1]{PRIM_IDEALS_DIFF_EQS}, for any $m \ge 0$ the map $\eta : R \hookrightarrow \BB$ induces an isomorphism of local rings $S/\pp^mS \xrightarrow{\cong} \BB/\nn^m \BB$.
	Accordingly, we obtain a commutative diagram 
	\begin{center}
		\begin{tikzpicture}[baseline=(current  bounding  box.center)]
			\matrix (m) [matrix of math nodes,row sep=2.5em,column sep=9em,minimum width=1.7em, text height=1.ex, text depth=0.25ex]
			{
				R^r &\\
				S^r &  \BB^r \\
				S^r/\pp^{m}S^r &  \BB^r/\nn^{m}\BB^r. \\
			};
			\path[-stealth]
			(m-3-1) edge node [above]	{$\cong$} (m-3-2)
			;		
			\path [draw,->>] (m-2-1) -- (m-3-1);
			\path [draw,->>] (m-2-2) -- (m-3-2);
			\draw[right hook->] (m-1-1)--(m-2-1);
			\draw[right hook->] (m-2-1)--(m-2-2);
			\draw[right hook->] (m-1-1)--(m-2-2) node [midway,above] {$\gamma$};
		\end{tikzpicture}	
	\end{center}
	By \autoref{rem_contain_power}, any $\BB$-submodule $\VV \subseteq \BB^r$ with $\dim_\FF(\BB^r/\VV) = m$ contains $\nn^m\BB^r$.
	Similarly, any $\pp S$-primary $S$-submodule $V \subseteq S^r$ with $\text{length}_S(S^r/V) = m$ contains $\pp^mS$.
	Therefore, the result follows because under the above identifications the constancy of multiplicity equal to $m$ will not change.
\end{proof}

We next recall the well-known Macaulay inverse systems for modules. 
Consider the injective hull $E := E_{\BB}(\FF)$ of the residue field $\FF \cong \BB/\nn$ of $\BB$.
Since $\BB$ is a formal power series ring, this can be identified with the set of inverse polynomials:
\begin{equation*}
	E \cong \FF[y_1^{-1},\ldots,y_c^{-1}]
\end{equation*}
(for more details see e.g. \cite[Lemma 11.2.3, Example 13.5.3]{Brodmann_Sharp_local_cohom}
or \cite[Theorem 3.5.8]{BRUNS_HERZOG}).
Let $\EE$ be the polynomial ring $\EE := \FF[z_1,\ldots,z_c]$, regarded as a $\BB$-module by letting $y_i$ act as $\partial_{z_i}$, i.e. $y_i \cdot F := \frac{\partial F}{\partial_{z_i}}$ for any $F \in \EE$.
Since $\FF$ has characteristic zero, there is an isomorphism of $\BB$-modules
\begin{equation*}
	E \cong \FF[y_1^{-1},\ldots,y_c^{-1}] \xrightarrow{\cong} \FF[z_1,\ldots,z_c] = \EE, \qquad \frac{1}{\mathbf{y}^{\alpha}} = \frac{1}{y_1^{\alpha_1}\cdots y_c^{\alpha_c}}  \mapsto  \frac{\mathbf{z}^\alpha}{\alpha!} = \frac{z_1^{\alpha_1}\cdots z_c^{\alpha_c}}{\alpha_1!\cdots \alpha_c!}.
\end{equation*}
We describe Macaulay inverse systems via Matlis duality.
Let ${\left(\bullet\right)}^\vee:=\Hom_{\BB}\left(\bullet,E\right)$ denote the Matlis dual functor (see e.g. \cite[Theorem 3.2.13]{BRUNS_HERZOG}).
For any $\BB$-submodule $\VV \subseteq \BB^r$, there is a natural identification
\begin{equation} \label{eq_V_perp}
\bigl(\BB^r/\VV \bigr)^\vee \cong \VV^{\perp} := \left\lbrace w \in \EE^r \mid  v \cdot w = \sum_{i=1}^{r}v_i \cdot w_i = 0 \mbox{ for all } v\in \VV \right\rbrace \subseteq \EE^r.
\end{equation}
On the other hand, any $\BB$-submodule $\WW \subseteq  \EE^r$
is an $\FF$-subspace of $\EE^r$
that is closed under differentiation, since $y_i$ acts as the operator $\partial_{z_i}$.
Thus we also obtain an identification
\begin{equation} \label{eq_W_perp}
\bigl(\EE^r/\WW \bigr)^\vee \cong \WW^{\perp} := \left\lbrace v \in \BB^r \mid  v \cdot w = \sum_{i=1}^{r}v_i \cdot w_i =0 \mbox{ for all } w \in \WW\right\rbrace \subseteq \BB^r.
\end{equation}
The above discussion then recovers the following well-known result.

\begin{theorem}[Macaulay inverse systems] \label{thm:Macaulay_dual}
	With the above notation, there is a bijection
	$$
	\begin{array}{ccc}
		\left\lbrace\begin{array}{c}
			\mbox{points in }{\rm Quot}^m\left(\BB^r\right)\\
		\end{array}\right\rbrace
		
		& \longleftrightarrow &
		
		\left\lbrace\begin{array}{c}
			\mbox{$m$-dimensional $\FF$-subspaces of}\\
			\mbox{$\EE^r$ closed under differentiation}
		\end{array}\right\rbrace\\
		
		\VV & \longrightarrow & \WW=\VV^\perp\\
		\VV= \WW^\perp & \longleftarrow & \WW.\\
	\end{array}
	$$
\end{theorem}

\subsection{The proof of the representation theorem}
\label{subsect_proof_represent}
In this subsection, we complete the proof of \autoref{thm_main}.
First, we recall some notation and results from \cite{PRIM_IDEALS_DIFF_EQS}.
Every differential operator $\delta \in D_{n,c}$ is a unique $\kk$-linear combination  of standard monomials  $\,\mathbf{x}^\alpha \partial_\mathbf{x}^\beta =  x_1^{\alpha_1} \cdots x_n^{\alpha_n}    \partial_{x_1}^{\beta_1} \cdots \partial_{x_c}^{\beta_c}$, where $\alpha_i, \beta_i \in \NN$.
Consider the \emph{Weyl-Noether module}
$$
\FF \otimes_R (D_{n,c})^r = \FF \otimes_R {\Diff_{R/\kk[x_{c+1},\ldots,x_n]}(R,R)}^r \cong \FF \otimes_S {\Diff_{S/\LL}(S,S)}^r
$$
(for the isomorphism on the right, see e.g. \cite[Lemma 2.7]{NOETH_OPS}).
From \autoref{prop_represen_diff_opp} and the fact that $P_{S/\LL}^m$ is a free $S$-module, the Weyl-Noether module admits the following description 
$$
\FF \otimes_R  (D_{n,c})^r =  \FF  \otimes_S  \biggl(\lim\limits_{\substack{\longrightarrow \\m}} {\Diff_{S/\LL}^m(S,S)}^r \biggr) \cong  \lim\limits_{\substack{\longrightarrow \\m}}
 \Diff_{S/\LL}^m(S,\FF)^r =  \Diff_{S/\LL}(S,\FF)^r.
$$
Applying  \autoref{lem_descript_diff_op} with $J=\pp S$ gives  $\FF \otimes_R D_{n,c} \cong \Diff_{S/\LL}(S,\FF) \cong \bigoplus_{\alpha \in \NN^c}\FF \overline{\partial_{\mathbf{x}}^\alpha}$. 
We then get an isomorphism of $\FF$-vector spaces:
\begin{equation*}
\omega : \EE = \FF[z_1,\ldots,z_c] \rightarrow \FF \otimes_R D_{n,c}, \quad \mathbf{z}^\alpha \mapsto \overline{\partial_{\mathbf{x}}^\alpha} \; \text{ for all } \alpha \in \NN^c,
\end{equation*}
which in turn induces an isomorphism
\begin{equation}
	\label{eq_map_omega}
	\Omega : \EE^r \to \FF \otimes_R (D_{n,c})^r.
\end{equation}
Let $\AAA$ be the polynomial ring $\AAA := \FF[y_1,\ldots,y_c]$.
Using \autoref{nota_T} over the field $\LL = \kk(x_{c+1},\ldots,x_n)$ with $T_S = S \otimes_\LL S =\LL[x_1,\ldots,x_c,y_1,\ldots,y_c]$ gives that $\AAA \cong \FF \otimes_\LL S  \cong \FF \otimes_S (S \otimes_\LL S) \cong \FF \otimes_S T_S$, and that $\FF \otimes_R (D_{n,c})^r \cong \Diff_{S/\LL}(S,\FF)^r \cong \Diff_{S/\LL}(S^r,\FF)$ has a natural structure of $\AAA$-module.
We now identify the power series ring $\BB=\FF[[y_1,\ldots,y_c]]$ of \autoref{subsect_Quot_Macaulay} with the completion $\widehat{\AAA_\nn}$ of $\AAA$ with respect to the maximal irrelevant ideal $\nn = (y_1,\ldots,y_c) \subseteq \AAA$; by an abuse of notation $\nn$ is seen interchangeably as an ideal in both $\AAA$ and $\BB$.
In this way, we get actions of $\AAA$ on both $\EE$ and $\FF \otimes_R D_{n,c}$.
Explicitly, for $\alpha \in \NN^c$ and $1 \le i \le c$,
\begin{equation*}
	y_i \cdot \mathbf{z}^\alpha := \alpha_iz_1^{\alpha_1} \cdots z_i^{\alpha_i-1} \cdots z_c^{\alpha_c} \quad \text{ and } \quad
	y_i \cdot \overline{\partial_{\mathbf{x}}^\alpha} := \left[ \overline{\partial_{\mathbf{x}}^\alpha}, x_i \right] =
	\alpha_i\overline{\partial_{x_1}^{\alpha_1}\cdots \partial_{x_i}^{\alpha_i-1} \cdots \partial_{x_c}^{\alpha_c}}.
\end{equation*}
Therefore the map $\Omega$ in \autoref{eq_map_omega} gives a bijection between $\FF$-vector subspaces of $\EE^r$ closed under differentiation and $\AAA$-submodules of $\FF \otimes_R  (D_{n,c})^r$. The latter
structure as an $\AAA$-submodule is equivalent to being an $R$-subbimodule of the Weyl-Noether module $\FF \otimes_R  (D_{n,c})^r$.

\begin{proof}[Proof of \autoref{thm_main}]
	The bijections \hyperref[part_a]{(a)} $\leftrightarrow$ \hyperref[part_b]{(b)} and \hyperref[part_b]{(b)} $\leftrightarrow$ \hyperref[part_c]{(c)} have been described in \autoref{thm:param_primary} and \autoref{thm:Macaulay_dual}, respectively.
	By the discussion above, the map $\Omega$
	in \autoref{eq_map_omega} provides the bijection \hyperref[part_c]{(c)} $\leftrightarrow$ \hyperref[part_d]{(d)}.
	Due to \autoref{lem_descript_diff_op}, we can lift differential operators from $\FF \otimes_R (D_{n,c})^r$ to $(D_{n,c})^r$ (cf. \cite[Remarks 7, 8]{PRIM_IDEALS_DIFF_EQS}).

	To finish the proof of the theorem, it suffices to show that an $\FF$-basis of the subspace in \hyperref[part_d]{(d)} lifts to a set of Noetherian operators for the submodule $U$ in \hyperref[part_a]{(a)}.
	That is,
	\begin{enumerate}
		\item let $U \subseteq R^r$ be a $\pp$-primary $R$-submodule of multiplicity $m$ over $\pp$,
		\item by \autoref{thm:param_primary} let $\VV := \BB\cdot \gamma(U) + \nn^m \BB^r \subseteq \BB^r$ be the corresponding point in $\Quot^m(\BB^r)$,
		\item by \autoref{thm:Macaulay_dual} let $\WW := \VV^\perp \subseteq \EE^r$ be the corresponding $m$-dimensional $\FF$-subspace of $\EE^r$ closed under differentiation,
		\item let $\mathcal{E} := \Omega(\WW) \subseteq \FF \otimes_R (D_{n,c})^r  \cong \Diff_{S/\LL}(S^r,\FF)$ be the corresponding $R$-subbimodule of $\FF \otimes_R (D_{n,c})^r$, 
	\end{enumerate}
	 then we claim $\Sol(\mathcal{E}) = U \otimes_R S$.
	
	Similarly to \cite[Lemma 3.14]{NOETH_OPS} and \cite[Proposition 3]{PRIM_IDEALS_DIFF_EQS}, the statements below hold:
	\begin{itemize}
		\item $\Diff_{S/\LL}^{m-1}(S^r, \FF) \cong \Hom_S\big(\text{P}_{S/\LL}^{m-1}(S^r), \FF\big) \cong \Hom_\FF\left(\AAA^r/\nn^{m}\AAA^r, \FF\right)$ by \autoref{prop_represen_diff_opp} and  Hom-tensor adjunction.
		\item Since $\VV \supseteq \nn^m\BB^r$ (cf. \autoref{rem_contain_power}) and $\BB^r / \nn^m\BB^r \cong \AAA^r/\nn^{m}\AAA^r$, we get $\mathcal{F} \subseteq \Diff_{S/\LL}^{m-1}(S^r, \FF)$  determined by $\Hom_\FF\left(\BB^r/\VV, \FF\right)$.
		Here we have $\mathcal{F} \cong \Hom_\FF\left(\BB^r/\VV, \FF\right) \subseteq \Hom_\FF\left(\BB^r/\nn^{m}\BB^r, \FF\right)$.
		\item $\Sol(\mathcal{F}) = U \otimes_R S$ (see \cite[Lemma 3.14(iv)]{NOETH_OPS} and \cite[Proposition 3(iii)]{PRIM_IDEALS_DIFF_EQS}).
	\end{itemize}
	
It thus suffices to show that $\mathcal{E}$ and $\mathcal{F}$ coincide as $R$-subbimodules of $\FF \otimes_R (D_{n,c})^r$.
By the perfect pairing of \cite[Proof of Theorem 6.1]{PRIM_IDEALS_DIFF_EQS} or general duality results (see e.g. \cite[Proposition 21.4]{EISEN_COMM}), there are isomorphisms

\begin{align} \label{eq_isom_large}
\begin{split}
	\mathcal{F} &\cong \Hom_\FF\left(\BB^r/\VV, \FF\right)   \\
		&\cong \Hom_{\BB}\big(\BB^r/\VV, \Hom_{\FF}(\BB/\nn^m,\FF)\big) \qquad (\text{by  Hom-tensor adjunction}) \\
		&\cong  \Hom_{\BB}\big(\BB^r/\VV, \Hom_{\BB}(\BB/\nn^m,E)\big) \qquad (\text{by \cite[Proposition 21.4]{EISEN_COMM}}) \\
		&\cong \Hom_\BB\left(\BB^r/\VV, E\right) = \left(\BB^r/\VV\right)^\vee \qquad \quad (\text{by  Hom-tensor adjunction}).
\end{split}
\end{align}
	Recall from \autoref{subsect_Quot_Macaulay} that the isomorphism $\left(\BB^r/\VV\right)^\vee \cong \VV^\perp = \WW$ is explicitly described by identifying the inverted monomial $\frac{1}{\yy^\alpha}=\frac{1}{y_1^{\alpha_1}\cdots y_c^{\alpha_c}}$ with $\frac{\mathbf{z}^\alpha}{\alpha!} = \frac{z_1^{\alpha_1}\cdots z_c^{\alpha_c}}{\alpha_1!\cdots \alpha_c!}$ for all $\alpha = (\alpha_1,\ldots,\alpha_c) \in \NN^c$.
	On the other hand, the isomorphism $\left(\BB^r/\VV\right)^\vee \cong \mathcal{F}$ is explicitly described by identifying the inverted monomial $\frac{1}{\yy^\alpha}$ with the dual monomial $\left(\yy^\alpha\right)^*$ and then with $\frac{1}{\alpha!}\overline{\partial_{\mathbf{x}}^\alpha}$ (see \autoref{rem_diff_yy});  notice  that we have an explicit isomorphism 
	$$
	\left(\BB/\nn^m\right)^\vee = \Hom_{\BB}(\BB/\nn^m, E) = \left(0:_{\FF[y_1^{-1},\ldots,y_c^{-1}]} \nn^m\right) \cong \Hom_\FF\left(\BB/\nn^m, \FF\right), \qquad \frac{1}{\yy^\alpha} \mapsto \left(\yy^\alpha\right)^*.
	$$
	Thus $\mathcal{E}$ and $\mathcal{F}$ do indeed coincide as $R$-subbimodules of $\FF \otimes_R (D_{n,c})^r$, as desired.
\end{proof}

 Finally, we have the following consequence. 

\begin{proof}[Proof of \autoref{cor_main}]
	From the short exact sequence $
	0 \to K \to R^r \to M \to 0
	$, 
	an $R$-submodule $U \subseteq M$ corresponds to a unique $R$-submodule $\widetilde{U} \subseteq R^r$ such that $\widetilde{U} \supset K$, and since $M/U \cong R^r/\widetilde{U}$, it follows that $\widetilde{U}$ is a $\pp$-primary submodule of $R^r$ of multiplicity $m$ over $\pp$.
	
	Let $\widetilde{U} \subseteq R^r$ be a $\pp$-primary submodule of multiplicity $m$ over $\pp$.
	By using the correspondences of \autoref{thm_main}, we set $\VV = \BB \cdot \gamma\big(\widetilde{U}\big) + \nn^m \BB^r$, $\WW = \VV^\perp$ and $\mathcal{E} = \Omega(\WW)$.
	In the same way, let $\VV' =\BB \cdot \gamma(
	K) +  \nn^m\BB^r$, $\WW' = (\VV')^\perp$ and $\mathcal{E}' = \Omega(\WW')$.
	Since the following four conditions are equivalent
	$$
	\widetilde{U} \supset K, \quad \VV \supset \VV', \quad \WW \subseteq \WW' \quad \text{and} \quad \mathcal{E} \subseteq \mathcal{E}',
 	$$
	the result follows directly from \autoref{thm_main}.
\end{proof}

\section{Differential algorithms} \label{sect_diff_algos}

In this section, we present several algorithms, based on \autoref{section_main_thm} and the previous papers \cite{PRIM_IDEALS_DIFF_EQS} and \cite{DIFF_PRIM_DEC}, that deal with the task of representing modules via differential operators. 
Together, they show that there currently exist powerful and increasingly versatile differential tools to represent modules computationally.
These algorithms are:

\begin{enumerate}[(I)]
	\item \autoref{algo_noeth_ops}: compute a set of Noetherian operators for a primary submodule.
	\item \autoref{algo_backwards}: compute the primary submodule determined by a set of differential operators. This can be seen as the inverse process to \autoref{algo_noeth_ops}.
	\item \autoref{algo_diff_prim_dec}: compute a minimal differential primary decomposition for a submodule.
\end{enumerate}


\subsection{Noetherian operators vs primary submodules}

This subsection deals with the problem of representing a primary submodule via Noetherian operators. 
We continue to use the notation of \autoref{section_main_thm}, cf. \autoref{setup_1}.
In particular, we continue to assume that $x_{c+1}, \ldots, x_n$ is independent modulo $\pp$ (notice that this can always be achieved after a suitable linear change of coordinates, and is also automatic if the module $U$ is zero-dimensional).
First, we give an algorithm to compute a set of Noetherian operators for a primary submodule.

\begin{algorithm}[Noetherian operators for a primary submodule]\label{algo_noeth_ops}
	\hfill\\
	{\sc Input:} A $\pp$-primary submodule $U \subseteq R^r$ of $R^r$ of multiplicity $m$ over  $\pp$.\\
	{\sc Output:}	A set of Noetherian operators $\delta_1,\ldots,\delta_m \in (D_{n,c})^r$ that represents $U$ as in \autoref{eq_noeth_ops}.
	\begin{enumerate}[(1)]
		\item Compute the $\FF[[y_1,\ldots,y_c]]$-module $\VV = \FF[[y_1,\ldots,y_c]] \cdot \gamma(U) + (y_1,\ldots,y_c)^m\FF[[y_1,\ldots,y_c]]^r$ that corresponds to $U$ as in \autoref{eq_corr_a_b}. 		
		\item  Using linear algebra over $\FF$, compute an  $\FF$-basis
		$\{B_1,\ldots,B_m\} \subseteq \FF[z_1,\ldots,z_c]^r$ for the inverse system $\WW = \VV^\perp$ as in \autoref{eq_corr_b_c}. 
		\item Compute $C_1 := \Omega(B_1), \ldots, C_m := \Omega(B_m) \in \FF \otimes_R (D_{n,c})^r$ as in \autoref{eq_corr_c_d}.
		\item Return lifts of $C_1,\ldots,C_m$ in $(D_{n,c})^r$, as guaranteed by \autoref{lem_descript_diff_op}. 
	\end{enumerate}
\end{algorithm}
\begin{proof}[Proof of correctness of \autoref{algo_noeth_ops}]
	The correctness of this algorithm follows from \autoref{thm_main}.
\end{proof}

In \autoref{algo_noeth_ops} the output is a set of Noetherian operators in the relative Weyl algebra $(D_{n,c})^r$.
We now consider the reverse process, starting from operators in the whole Weyl algebra $D_n$.
We start with some basic facts regarding modules defined via differential operators.

\begin{remark}
	\label{rem_sol_bimod_gen}
	Let $\delta_1,\ldots,\delta_m \in \DiffR(R^r, R) = (D_n)^r$ be differential operators.
	Let $\mathcal{G} \subseteq (D_n)^r$ be the $R$-bimodule generated by $\delta_1,\ldots,\delta_m$.
	The following statements hold: 
	\begin{enumerate}[(i)]
		\item $\lbrace w \in R^r \mid \delta(w) \in \pp \text{ for all } \delta \in \mathcal{G}\rbrace$ is a $\pp$-primary $R$-submodule of $R^r$.
		\item $\lbrace w \in R^r \mid \delta(w) \in \pp \text{ for all } \delta \in \mathcal{G}\rbrace \subseteq \lbrace w \in R^r \mid \delta_i(w) \in \pp \text{ for all } 1 \le i \le m\rbrace$, and equality holds if and only if the right hand side
is an $R$-submodule of $R^r$.
	\end{enumerate}
\end{remark}
\begin{proof}
	For more details, see \cite[\S 3]{NOETH_OPS} and specifically \cite[Proposition 3.5]{NOETH_OPS}.
\end{proof}

\begin{remark}
	It is a challenging problem to decide when  $\lbrace w \in R^r \mid \delta_i(w) \in \pp \text{ for all } 1 \le i \le m\rbrace$ is an $R$-module for a given list $\delta_1,\ldots,\delta_m \in \DiffR(R^r, R) = (D_n)^r$ of differential operators. 
	This question was addressed in \cite[Theorem 3.1]{PRIM_IDEALS_DIFF_EQS}, and necessary and sufficient conditions were given there.
	However, as stated in \cite[Paragraph after Theorem 3.1]{PRIM_IDEALS_DIFF_EQS}, we currently do not have a good method to verify these conditions.
\end{remark}

In view of \autoref{rem_sol_bimod_gen}, it is desirable to treat the following ``closure operation'': given finitely many differential operators  $\delta_1,\ldots,\delta_m \in (D_n)^r$, compute the corresponding $\pp$-primary $R$-submodule 
$$
\lbrace w \in R^r \mid \delta(w) \in \pp \text{ for all } \delta \in \mathcal{G}\rbrace,
$$ where $\mathcal{G} \subseteq (D_n)^r$ is the $R$-bimodule generated by $\delta_1,\ldots,\delta_m$.

The subsequent arguments follow verbatim the techniques used in \autoref{section_main_thm}.
 Here we use the whole sets of variables $y_1,\ldots,y_n$ and $z_1,\ldots,z_n$ instead of just $y_1,\ldots,y_c$ and $z_1,\ldots,z_c$, respectively.
 The only issue with taking whole sets of variables is that \autoref{thm:param_primary} is no longer valid (as \cite[Proposition 1]{PRIM_IDEALS_DIFF_EQS} requires $\FF/\kk(x_{c+1}, \ldots, x_n)$ to be algebraic), but we may circumvent this via \autoref{rem_sol_bimod_gen}.
We have a canonical map 
\begin{equation}
	\label{eq_map_Phi}
	\Phi : \DiffR(R^r, R) = (D_n)^r \rightarrow \FF \otimes_R (D_{n})^r.
\end{equation}
Following \autoref{nota_T}, by \autoref{prop_represen_diff_opp}  (see \autoref{rem_diff_yy}) we obtain the isomorphism
\begin{equation}
	\label{eq_isom_Diff_dual}
	\FF \otimes_R \DiffR^m(R^r,R) \cong \FF \otimes_R \Hom_R\big(P_{R/\kk}^m(R^r), R\big)  \cong  \Hom_{\FF}\left(\frac{\AAA^r}{\nn^{m+1}\AAA^r}, \FF\right)
\end{equation}
where $\AAA = \FF[y_1,\ldots,y_n]$ and $\nn = (y_1,\ldots,y_n)$.
Let $\BB=\FF[[y_1,\ldots,y_n]]$ and $\EE = \FF[z_1,\ldots,z_n]$, and as before, consider $\EE$ as a $\BB$-module by setting $y_i = \partial_{z_i}$ for all $1 \le i \le n$.
As is \autoref{eq_V_perp} and \autoref{eq_W_perp}, we can define $\VV^\perp$ and $\WW^\perp$ for $\VV \subseteq \BB^r = \FF[[y_1,\ldots,y_n]]^r$ and $\WW \subseteq \EE^r = \FF[z_1,\ldots,z_n]^r$, respectively.
In this setting, Macaulay inverse systems (\autoref{thm:Macaulay_dual}) are also valid.
Now, the equivalent map of $\gamma$ in \autoref{eq_map_gamma} is given by 
\begin{equation}
	\label{eq_map_Gamma}
	\Gamma : R^r \hookrightarrow \BB^r = \FF[[y_1,\ldots,y_n]]^r, \quad (f_1,\ldots,f_r) \in R^r \mapsto (\eta'(f_1),\ldots,\eta'(f_r)) \in \BB^r
\end{equation}
where $\eta' : R \hookrightarrow \BB$, $x_i \mapsto y_i + \overline{x_i}$ for all $1 \le i \le n$, and the equivalent of the map $\Omega$ in \autoref{eq_map_omega} is given by 
\begin{equation}
	\label{eq_map_Psi}
	\Psi : \EE^r = \FF[z_1,\ldots,z_n]^r \rightarrow \FF \otimes_R (D_{n})^r
\end{equation}
and is induced by $\mathbf{z}^\alpha \mapsto \overline{\partial_{\mathbf{x}}^\alpha}$ for all $\alpha \in \NN^n$ (in this setting, under the isomorphism  \autoref{eq_isom_Diff_dual} with $r = 1$ and $|\alpha| \le m$, the dual monomial $\left(\yy^\alpha\right)^*$ coincides with $\overline{\partial_{\mathbf{x}}^\alpha} \in \FF\otimes_R D_n$).

After the above discussion, we can present the following algorithm which can be seen as an inverse process to \autoref{algo_noeth_ops}.

\begin{algorithm}[Primary submodule that corresponds to a set of differential operators] \label{algo_backwards}\hfill\\
	{\sc Input:} A prime ideal $\pp \in \Spec(R)$ and a set of differential operators $\delta_1,\ldots,\delta_m \in (D_n)^r$.\\
	{\sc Output:}	The $\pp$-primary $R$-submodule $\lbrace w \in R^r \mid \delta(w) \in \pp \text{ for all } \delta \in \mathcal{G}\rbrace$ of $R^r$, where $\mathcal{G} \subseteq (D_n)^r$ is the $R$-bimodule generated by $\delta_1,\ldots,\delta_m$.
	\begin{enumerate}[(1)]
		\item By using \autoref{eq_map_Phi} and \autoref{eq_map_Psi}, compute $A_1 := \Psi^{-1}(\Phi(\delta_1)), \ldots, A_m := \Psi^{-1}(\Phi(\delta_m)) \in  \FF[z_1,\ldots,z_n]^r$.
		\item  Using linear algebra over $\FF$, compute the inverse system 
		$$
		\VV := \left(A_1,\ldots,A_m\right)^\perp = \Big\lbrace v \in \FF[[y_1,\ldots,y_n]]^r \mid  v \cdot A_i = 0 \mbox{ for all }
		1 \le i \le m\Big\rbrace. 
		$$
		\item  Return the $R$-submodule $U := \Gamma^{-1}(\VV) \subseteq R^r$ where $\Gamma$ is the map in \autoref{eq_map_Gamma}.
	\end{enumerate}
\end{algorithm}
\begin{proof}[Proof of correctness of \autoref{algo_backwards}]
	Let $\mathcal{G} \subseteq  (D_n)^r$ be the $R$-bimodule generated by $\delta_1,\ldots,\delta_m$, and let $\mathcal{E} = \Phi(\mathcal{G}) \subseteq \FF \otimes_R (D_n)^r$. 
	Then $\WW = \Psi^{-1}(\mathcal{E}) \subseteq \EE^r =  \FF[z_1,\ldots,z_n]^r$ is an $\FF$-vector subspace that is closed under differentiation (the same arguments of \autoref{eq_corr_c_d}).
	Also the equalities $\VV = \WW^\perp \subseteq  \BB^r$ and $\WW = \VV^\perp \subseteq \EE^r$ hold.
	In a similar way to the proof of \autoref{thm_main} (see \autoref{eq_isom_large}), we obtain
$\mathcal{E}  \cong  \Hom_{\FF}\left(\BB^r/\VV, \FF\right)$.
	By the same general arguments of \cite[Lemma 3.14]{NOETH_OPS} and \cite[Proposition 3]{PRIM_IDEALS_DIFF_EQS}, it follows that 
	$\lbrace w \in R^r \mid \delta(w) \in \pp \text{ for all } \delta \in \mathcal{G}\rbrace = \Sol(\mathcal{E}) = \Gamma^{-1}(\VV) = U$, which is $\pp$-primary by \autoref{rem_sol_bimod_gen}.
\end{proof}

\begin{remark}
	When we consider whole sets of variables $y_1,\ldots,y_n$ and $z_1,\ldots,z_n$, we can describe primary submodules in quite different ways. 
	This obstacle tells us that without restricting the variables, the best extension of \autoref{thm_main} that we can obtain is that: we can always recover a primary submodule from any of its possibly many ways to encode it with Noetherian operators.
	In \autoref{thm_main} we obtained a \emph{bijection} between four different sets of objects, but if we allow whole sets of variables a bijection is not attainable.
	For instance, see \cite[Example 10]{PRIM_IDEALS_DIFF_EQS}, where a primary ideal of codimension two in $\kk[x_1,\ldots,x_4]$ is represented with two quite different sets of Noetherian operators. 
	One of them is given by sixteen differential operators with constant coefficients in $D_4$, and the other is given by eleven differential operators (some of them with polynomial coefficients) in $D_{4,2} \subset D_4$.
\end{remark}

\subsection{Differential primary decomposition for modules}
\label{subsect_diff_prim_dec}
The main goal of this subsection is to provide an algorithm that computes a \emph{minimal differential primary decomposition} for general modules.
The concept of a \emph{differential primary decomposition} was introduced in \cite{DIFF_PRIM_DEC}, and an algorithm for the case of ideals was given in \cite[Algorithm 5.4]{DIFF_PRIM_DEC}.


First, we extend the notation in \autoref{setup_1}, for multiple primes $\pp_i$.
 For a set of variables $\mathcal{S} = \{x_{i_1},  \ldots, x_{i_\ell} \} \subseteq \{x_1,\ldots,x_n\}$, consider the corresponding \emph{relative Weyl algebra},
 which consists of $\KK[\mathcal{S}]$-linear differential operators on $R$:
 $$
 D_{n}(\mathcal{S}) := \Diff_{R/\KK[\mathcal{S}]}(R,R) = R \big< \partial_{x_i} \mid x_i \not\in \mathcal{S} \big> \subseteq  R\langle \partial_{x_1},\ldots, \partial_{x_n} \rangle = D_n.
 $$
 We now recall the definition of the main object of interest in this section.

\begin{definition}[{cf. \cite[Definition 5.2, Definition 4.1]{DIFF_PRIM_DEC}}] \label{def:diffPrimDec}
	Let $U \subseteq R^r$ be an $R$-module,
	with associated primes $\Ass(R^r/U) =: \{\pp_1,\ldots,\pp_k\}$.
	A \emph{differential primary decomposition} of $U$ is a list of triples
	$$(\pp_1, \mathcal{S}_1, \AAA_1), \; (\pp_2, \mathcal{S}_2, \AAA_2), \; \ldots, \; (\pp_k, \mathcal{S}_k, \AAA_k)$$
	where $\mathcal{S}_i$ is a basis modulo $\pp_i$ and
	$\AAA_i \subseteq D_{n}(\mathcal{S}_i)^r$ is a finite set of differential operators such that
	$$ 
	\qquad
	U_\pp \cap R^r \,\,=\; \bigcap_{i: \pp_i \subseteq \pp} \!\!
	\big\lbrace w \in R^r \mid \delta(w) \in \pp_i \text{ for  all }   \delta \in \AAA_i  \big\rbrace \quad
	\text{for each $\pp \in \Ass(R^r/U)$.}
	$$
	These conditions imply that 
	$
	U \;= \; \big\lbrace w \in R^r \mid \delta(w) \in \pp_i \text{ for  all }  \delta \in \AAA_i
	\text{ and } 1 \le i \le k \big\rbrace
	$.
	The \emph{size} of the differential primary decomposition is defined to be $\sum_{i=1}^k |\AAA_i|$.
\end{definition}

\begin{definition}[{\cite{STV_DEGREE}, \cite[Definition 4.3]{DIFF_PRIM_DEC}}] \label{def:amult}
	For an $R$-submodule $U \subseteq R^r$, its \emph{arithmetic multiplicity} 
	is the non-negative integer
	$$ \amult(U) \,\, :=  \sum_{\pp \in \Ass(R^r/U)} \!\!\!
	\text{length}_{R_\pp}\left(\HH_{\pp}^0\Big(R^r_\pp/U_\pp\Big)\right) 
	\quad = \sum_{\pp \in \Ass(R^r/U)} \!\!\! \text{length}_{R_\pp}\left(
	\frac{\left(U_\pp:_{R^r_\pp} {(\pp R_\pp)}^{\infty}\right)}{U_\pp} \right).
	$$
\end{definition}

In \cite{DIFF_PRIM_DEC}, it was shown that the arithmetic multiplicity is the minimal size for a differential primary decomposition of a module (see \cite[Theorems 3.6, 4.6, 5.3]{DIFF_PRIM_DEC}).
We point out that a differential primary decomposition \emph{should not} be thought of as a notion directly related to an ordinary primary decomposition.
In \autoref{def:diffPrimDec}, a differential primary decomposition is carefully designed to recover the localizations of a module along its associated primes; recall that the embedded components of a primary decomposition are not an invariant of a module, but the localizations at the associated primes are.
Given an $R$-module $U \subseteq R^r$ and a differential primary decomposition $(\pp_1, \mathcal{S}_1, \AAA_1),  \ldots, (\pp_k, \mathcal{S}_k, \AAA_k)$,  the $\kk$-vector space $\lbrace w \in R^r \mid \delta(w) \in \pp_i \text{ for  all }   \delta \in \AAA_i  \rbrace$ may not be an $R$-module for some $1 \le i \le k$.
Below we have an example signaling this important phenomenon.

\begin{example}
	\label{examp_dumb}
	Let $R=\kk[x_1,x_2]$  and $I = (x_1^2, x_1x_2) \subset R$ be an ideal.
	The associated primes of $I$ are $\pp_1 = (x_1)$ and $\pp_2  = (x_1,x_2)$.
	A minimal differential primary decomposition of $I$ is given by the two triples $(\pp_1,\{x_1\}, \{1\} )$ and $(\pp_2,\{\}, \{\partial_{x_1}\} )$, and the arithmetic multiplicity of $I$ is $\amult(I)=2$.
	We have the following equality
	$$
	I = \{f \in R \mid  f \in \pp_1 \} \;\cap\;  \{f \in R \mid  \tfrac{\partial f}{\partial x_1} \in \pp_2 \}.
	$$	
	Notice that $V_2 = \{f \in R \mid  \tfrac{\partial f}{\partial x_1} \in \pp_2 \}$ is the $\kk$-vector space of polynomials $f \in R$ such that $x_1$ does not appear in the expansion of $f$.
	Therefore, since $1 \in V_2$ and yet $x_1 \not\in V_2$, it follows that $V_2$ is not an $R$-module.
	
	Finally, an ordinary primary decomposition of $I$ is given by $I = \pp_1 \cap (x_1^2, x_2)$, where the primary components are described as $\pp_1 = \{f \in R \mid  f \in \pp_1 \}$ and $(x_1^2, x_2) = \{f \in R \mid f \in \pp_2 \text{ and }  \tfrac{\partial f}{\partial x_1} \in \pp_2 \}$ in terms of Noetherian operators --	notice that this gives the \emph{non-minimal} differential primary decomposition $(\pp_1,\{x_1\}, \{1\} )$ and $(\pp_2,\{\}, \{1, \partial_{x_1}\} )$.	
\end{example}

By using our representation theorem (\autoref{thm_main}), we can extend \cite[Algorithm 5.4]{DIFF_PRIM_DEC} from ideals to modules.
The following algorithm computes a minimal differential primary decomposition for a module.
We note that this algorithm does not depend on computing a primary decomposition.

\begin{algorithm}[Differential primary decomposition for modules]\label{algo_diff_prim_dec}
	\hfill\\
	{\sc Input:} An $R$-submodule $U \subseteq R^r$. \\
	{\sc Output:}	A differential primary decomposition for $U$ of minimal size ${\rm amult}(U)$.
	
	\begin{enumerate}[(1)]
		\item Compute the set of associated primes $\Ass(R^r/U) =: \{\pp_1,\ldots,\pp_k\}$, e.g. via \autoref{thm:associatedPrimesEHV}.
		\item For $i$ from $1$ to $k$ do:
		\begin{enumerate}[(2.1)]
			\item Compute a basis $\mathcal{S}_i$ modulo $\pp_i$, and let $\FF_i := k(\pp_i)$ be the residue field of $\pp_i$.
			\item Compute $\mathcal{U} := U_{\pp_i} \cap R^r$ -- this is the intersection of
			all primary components of $U$ whose associated prime is contained in $\pp_i$, e.g. via \autoref{prop:localPrimaryComponents}.
			\item Compute $\mathcal{V} := \mathcal{U} :_{R^r} \pp_i^\infty $ -- this is the intersection of
			all  primary components of $U$ whose associated prime is strictly contained in $\pp_i$, e.g. via \autoref{prop:localPrimaryComponents}.
			\item Find $m >0$ giving the isomorphism in \cite[Proposition 4.5]{DIFF_PRIM_DEC}(i):
			$$
			\mathcal{V}/\mathcal{U}  \xrightarrow{\cong} (\mathcal{V} + {\pp_i}^mR^r)/(\mathcal{U} + {\pp_i}^mR^r).
			$$
			\item  By using \autoref{thm_main} and \autoref{algo_noeth_ops}, compute the $\FF_i$-vector subspaces $\mathcal{E}$ and $\mathcal{H}$ of the Weyl-Noether module 
			$\FF_i \otimes_R D_n(\mathcal{S}_i)^r$.
			These are $(R \otimes_{\KK[\mathcal{S}_i]} R)$-modules that correspond to the $\pp_i$-primary submodules $\mathcal{U} + \pp_i^mR^r$ and $\mathcal{V} + \pp_i^mR^r$, respectively.
			\item Compute an $\FF_i$-basis $\overline{\AAA_i}$ of an $\FF_i$-vector space complement $\mathcal{G}$ of $\mathcal{H}$ in $\mathcal{E}$, i.e. $\mathcal{E} = \mathcal{H} \oplus \mathcal{G}$.
			\item Lift the basis $\overline{\AAA_i}$ to a subset $\AAA_i \subseteq D_n(\mathcal{S}_i)^r$.
		\end{enumerate}
		\item 
		Return the triples $(\pp_1, \mathcal{S}_1, \AAA_1)$, $\ldots$, $(\pp_k, \mathcal{S}_k, \AAA_k)$.
	\end{enumerate}
\end{algorithm}
\begin{proof}[Proof of correctness of \autoref{algo_diff_prim_dec}]
	This algorithm is correct because it realizes the steps in the proof of \cite[Theorem 5.3 (i)]{DIFF_PRIM_DEC}, as generalized by \autoref{thm_main} and \autoref{algo_noeth_ops}.
	(For the case of ideals, see \cite[Algorithm 5.4]{DIFF_PRIM_DEC}.)
\end{proof}

\begin{remark}
	\label{rem_dumb}
	The reverse process of \autoref{algo_diff_prim_dec} can also be achieved.
	If we are given a differential primary decomposition $(\pp_1, \mathcal{S}_1, \AAA_1)$, $\ldots$, $(\pp_k, \mathcal{S}_k, \AAA_k)$ as the output of \autoref{algo_diff_prim_dec}, then via \autoref{algo_backwards} we can compute the $\pp_i$-primary submodule $U_i \subseteq R^r$ that corresponds to $\AAA_i$, so that $U = \bigcap_{i=1}^k U_i$.
	In this way a differential primary decomposition gives a primary decomposition as in \autoref{sect_symb_prim_dec}.
	In general though, this ``dual'' algorithm for primary decomposition will be slower than the algorithm given in \autoref{sect_symb_prim_dec}, as the bulk of the computation  in \autoref{algo_diff_prim_dec}  is subsumed in step (2.5) which entails heavy computations over  the residue field $\FF_i$.
\end{remark}

\section{A differential description of joins}
\label{sect_joins}

In this short section, we provide an intrinsic differential description of the ideal join of an ideal and a primary ideal with respect to the maximal irrelevant ideal. 	
This result was implicitly obtained in the proof of \cite[Theorem 7.1]{PRIM_IDEALS_DIFF_EQS}, but the statement was not explicitly given as the objective was to get an extension of a result of Sullivant \cite[Proposition~2.8]{SULLIVANT_SYMB}.

Let $\kk$ be a field of characteristic zero, $R = \kk[x_1,\ldots,x_n]$ be a polynomial ring and $\mm = (x_1,\ldots,x_n) \subseteq R$ be the maximal irrelevant ideal. 
If $J$ and $K$ are ideals in $R$, then their \textit{join} is given by the new ideal
$$
J \star K \,\,\,:=\,\,\, \Big( J(\mathbf{v}) \,+\, K(\mathbf{w}) \,+\,
\left( x_i - v_i - w_i \mid 1 \le i \le n \right) \Big) \,\,\cap \,\,R,
$$
where $J(\mathbf{v}) \subseteq \kk[v_1,\ldots,v_n]$ is the ideal $J$ with new variables $v_i$ substituted for $x_i$ and $K(\mathbf{w}) \subseteq \kk[w_1,\ldots,w_n]$ is the ideal $K$ with  $w_i$ substituted for $x_i$.	

The next definitions are obtained by mimicking  \cite[Definition 4]{PRIM_IDEALS_DIFF_EQS}.

\begin{definition}
	Let $M \subseteq R$ be an $\mathfrak{m}$-primary ideal.
	We compute a finite dimensional $\kk$-subspace $\mathfrak{A}(M) \subseteq \kk[\partial_{x_1},\ldots,\partial_{x_n}]$ of differential operators with constant coefficients by performing the following~steps:
	\begin{enumerate}[(i)]
		\item Interpret the variable $x_i$ as $\partial_{z_i}$
		for $i=1,\ldots,n$. 
		\item Compute $M^\perp=\left\lbrace F\in \kk[z_1,\ldots,z_n] \mid f\cdot F=0 \mbox{ for all }f\in M\right\rbrace$. 
		\item Let $\mathfrak{A}(M) \subseteq \kk[\partial_{x_1},\ldots,\partial_{x_n}]$ be the image of $M^\perp$ under the map $\mathbf{z}^\alpha \mapsto \partial_{\mathbf{x}}^\alpha$.
	\end{enumerate}
\end{definition}

A $\kk$-subspace $V \subseteq \kk[\partial_{x_1}, \ldots,\partial_{x_n}]$ is said to be \emph{closed under the bracket operation} if $[\delta, x_i] \in V$ for all $\delta \in V$ and $1 \le i \le n$.
Recall that $\left[\partial_{x_1}^{\alpha_1}\cdots\partial_{x_i}^{\alpha_i} \cdots \partial_{x_n}^{\alpha_n}, \, x_i\right] = \alpha_i \partial_{x_1}^{\alpha_1}\cdots\partial_{x_i}^{\alpha_i-1} \cdots \partial_{x_n}^{\alpha_n}$ for all $\alpha \in \NN^n$ and $1 \le i \le n$.

\begin{definition}
	Let $V \subseteq \kk[\partial_{x_1},\ldots,\partial_{x_n}]$ be a finite dimensional $\kk$-subspace closed under the bracket operation.
	We compute an $\mm$-primary ideal  $\mathfrak{B}(V) \subseteq R$  by performing the following~steps:
	\begin{enumerate}[(i)]
		\item Interpret the variable $x_i$ as $\partial_{z_i}$
		for $i=1,\ldots,n$. 
		\item Let $W \subseteq \kk[z_1,\ldots,z_n]$ be the image of $V$ under the map $\partial_{\mathbf{x}}^\alpha \mapsto \mathbf{z}^\alpha$.
		\item Let  $\mathfrak{B}(V) = W^\perp=\left\lbrace f \in R \mid f\cdot F=0 \mbox{ for all } F\in W\right\rbrace$. 
	\end{enumerate}
\end{definition}

The following result can be easily deduced from \cite[Theorem 7.1]{PRIM_IDEALS_DIFF_EQS}.

\begin{theorem}
	\label{thm_join}
	Let $J \subseteq R$ be an ideal.
	Then, the following statements hold: 
	\begin{enumerate}[(i)]
		\item If $M \subseteq R$ is an $\mm$-primary ideal, then 
		$
		J \star M \; = \;  \big\lbrace f \in R \mid \delta \cdot f \in J \;
		\text{ for all }\; \delta \in \mathfrak{A}(M) \big\rbrace. 
		$
		\item If $V \subseteq \kk[\partial_{x_1},\ldots,\partial_{x_n}]$ is a finite dimensional $\kk$-subspace closed under the bracket operation, then 
		$
		\big\lbrace f \in R \mid \delta \cdot f \in J \;
		\text{ for all }\; \delta \in V \big\rbrace \; = \; J \star \mathfrak{B}(V).
		$
	\end{enumerate}
\end{theorem}
\begin{proof}
	(i) This is the statement of \cite[Theorem 7.1 (i)]{PRIM_IDEALS_DIFF_EQS}.
	
	(ii) Let $V \subseteq \kk[\partial_{x_1},\ldots,\partial_{x_n}]$ be a finite dimensional $\kk$-subspace closed under the bracket operation, and set $M$ to be the corresponding $\mm$-primary ideal $M = \mathfrak{B}(V)$.
	Notice that $V = \mathfrak{A}(M)$ and $M = \mathfrak{B}(V)$, as both $\mathfrak{A}(M)$ and $\mathfrak{B}(V)$ are computed via Macaulay inverse systems.
	By \cite[Theorem 7.1 (i)]{PRIM_IDEALS_DIFF_EQS}, we obtain 
	$
	\big\lbrace f \in R \mid \delta \cdot f \in J \;
	\text{ for all }\; \delta \in V \big\rbrace \; = \; J \star M,
	$
	and so the result follows.
\end{proof}

Of particular interest is the case when $J$ is a prime ideal.
Let $\pp \in \Spec(R)$ be a prime ideal. 
A $\pp$-primary ideal $Q \subseteq R$ is said  to be \emph{representable by differential operators with constant coefficients} if there exist $\delta_1,\ldots,\delta_m \in \kk[\partial_{x_1},\ldots,\partial_{x_n}]$ such that $Q = \lbrace f \in R \mid \delta_i \cdot f \in \pp \text{ for all } 1 \le i \le m\rbrace$.
A $\pp$-primary $Q \subseteq R$ is said to be \emph{representable by the join construction} if there exists an $\mm$-primary ideal $M \subseteq R$ such that $Q = \pp \star M$.

\begin{corollary}
	Let $\pp \in \Spec(R)$ be a prime ideal and $Q \subseteq R$ be a $\pp$-primary ideal.
	Then, $Q$ is representable by differential operators with constant coefficients if and only if $Q$ is representable by the join construction.
\end{corollary}
\begin{proof}
	It follows from \autoref{thm_join}.
\end{proof}

\section{Examples} \label{sect_computations}

We conclude by demonstrating the algorithms in \autoref{sect_symb_prim_dec} and \autoref{sect_diff_algos} on some examples. 
The main commands in our implementation are: 
\begin{itemize}[--]
	\item \texttt{primaryDecomposition(Module)}: executes an implementation of the algorithm in \autoref{sect_symb_prim_dec}.
	\item \texttt{noetherianOperators}: executes an implementation of \autoref{algo_noeth_ops}.
	\item \texttt{getModuleFromNoetherianOperators}: executes an implementation of \autoref{algo_backwards}.
	\item \texttt{differentialPrimaryDecomposition}: executes an implementation of \autoref{algo_diff_prim_dec}.
\end{itemize}

\begin{remark}
The algorithms above have been implemented in \texttt{Macaulay2} \cite{MACAULAY2}.
\texttt{primaryDecomposition} is part of the default package \texttt{PrimaryDecomposition}, and the algorithm for modules is available to use as of version $1.17$.
The differential algorithms will appear in the package \texttt{NoetherianOperators} \cite{CCHKL} starting from version $1.18$.
For convenience, we have included these functions in a separate ancillary file ``\texttt{modulesNoetherianOperators.m2}''.
\end{remark}

We start with a simple example that first appeared in \cite[Example 4.4]{DIFF_PRIM_DEC}.

\begin{example}
	\label{examp_simple}
	Let  $R = \mathbb{Q}[x_1,x_2,x_3]$ and $U \subseteq R^2$ be the $R$-submodule
	$U = {\rm image}_R \! \begin{small} \begin{bmatrix} x_1^2 & x_1x_2 & x_1x_3 \\
			x_2^2 & x_2x_3 & x_3^2 \end{bmatrix}\end{small}$.
	We compute a primary decomposition and a minimal differential primary decomposition for $U$:
	\verbatimfont{\footnotesize}
	\begin{verbatim}
Macaulay2, version 1.17.2.1
i1 : load "modulesNoetherianOperators.m2";
i2 : printPD = M -> apply(primaryDecomposition M, Q -> trim image(gens Q | relations Q));
i3 : R = QQ[x_1,x_2,x_3];
i4 : U = image matrix {{x_1^2,x_1*x_2,x_1*x_3}, {x_2^2,x_2*x_3,x_3^2}};
i5 : M = R^2 / U;
i6 : L1 = printPD M
o6 = {image | 0 x_1 |, image | x_1 x_2^2 0            |, image | x_3 x_2^2 0     x_1x_2 x_1^2 |}
            | 1 0   |        | x_3 x_3^2 x_2^2-x_1x_3 |        | 0   0     x_3^2 x_2x_3 x_2^2 |
o7 : all(L1, isPrimary_M) and U == intersect L1
o7 = true
i8 : L2 = differentialPrimaryDecomposition U
                                   2
o8 = {{ideal x , {| 1 |}}, {ideal(x  - x x ), {| -x_3 |}}, {ideal (x , x ), {| 0    |}}}
              1   | 0 |            2    1 3    | x_1  |             3   2    | dx_3 |
o8 : List
i9 : U == intersect apply(L2, getModuleFromNoetherianOperators)
o9 = true
i10 : amult U
o10 = 3
	\end{verbatim}
	Notice that $\amult(U) = 3$ is the size of the computed differential primary decomposition.
\end{example}


\begin{example}
	Let  $R = \mathbb{Q}[x_1,x_2,x_3,x_4]$ and $U \subseteq R^2$ be the $R$-submodule
	$$
	U \;=\; {\rm image}_R \! 
	\begin{small} 
		\begin{bmatrix}  
			x_1x_2 & x_2x_3 & x_3x_4 & x_4x_1 \\
			x_1^2   & x_2^2   & x_3^2   & x_4^2
		\end{bmatrix}
	\end{small}.
	$$
\verbatimfont{\tiny}
\begin{verbatim}
i11 : R = QQ[x1,x2,x3,x4];
i12 : U = image matrix{{x1*x2,x2*x3,x3*x4,x4*x1}, {x1^2,x2^2,x3^2,x4^2}};
i13 : M = R^2 / U;
i14 : L1 = printPD M;
i15 : all(L1, isPrimary_M) and U == intersect L1
o15 = true
i16 : netList transpose{associatedPrimes M, L1}

      +-----------------------------------+---------------------------------------------------------------------------+
o16 = |ideal (x4, x2)                     |image | 0 x4 x2 |                                                          |
      |                                   |      | 1 0  0  |                                                          |
      +-----------------------------------+---------------------------------------------------------------------------+
      |ideal (x3, x1)                     |image | 0 x3 x1 |                                                          |
      |                                   |      | 1 0  0  |                                                          |
      +-----------------------------------+---------------------------------------------------------------------------+
      |ideal (x4, x3, x2)                 |image | x4 0  x2 0    x2x3 x3^6 |                                          |
      |                                   |      | 0  x4 x1 x3^2 x2^2 0    |                                          |
      +-----------------------------------+---------------------------------------------------------------------------+
      |ideal (x4, x3, x1)                 |image | x1 x3 0  0    x3x4 x4^6 |                                          |
      |                                   |      | 0  x2 x1 x4^2 x3^2 0    |                                          |
      +-----------------------------------+---------------------------------------------------------------------------+
      |ideal (x4, x2, x1)                 |image | x2 x4 0  x1x4 0    x1^6 |                                          |
      |                                   |      | 0  x3 x2 x4^2 x1^2 0    |                                          |
      +-----------------------------------+---------------------------------------------------------------------------+
      |ideal (x3, x2, x1)                 |image | x3 x1 0  0    0         x2^6 |                                     |
      |                                   |      | 0  x4 x3 x2^2 x1^2-x2x4 0    |                                     |
      +-----------------------------------+---------------------------------------------------------------------------+
      |ideal (x3 - x4, x2 - x4, x1 - x4)  |image | 1 x3-x4 x2-x4 x1-x4 |                                              |
      |                                   |      | 1 0     0     0     |                                              |
      +-----------------------------------+---------------------------------------------------------------------------+
      |ideal (x3 + x4, x2 - x4, x1 + x4)  |image | -1 x3+x4 x2-x4 x1+x4 |                                             |
      |                                   |      | 1  0     0     0     |                                             |
      +-----------------------------------+---------------------------------------------------------------------------+
      |                           2     2 |                                                                           |
      |ideal (x2 + x4, x1 + x3, x3  + x4 )|image | x2+x4 x3  x1 x4 0     -x4 |                                        |
      |                                   |      | 0     -x4 x4 x3 x2+x4 x1  |                                        |
      +-----------------------------------+---------------------------------------------------------------------------+
      |ideal (x4, x3, x2, x1)             |image | x1x4 x3x4 x2x3 x1x2 x4^4 x3^3x4 x3^4 x2^3x3 x2^4 x1^3x2 x1^4 0    ||
      |                                   |      | x4^2 x3^2 x2^2 x1^2 0    0      0    0      0    0      0    x4^4 ||
      +-----------------------------------+---------------------------------------------------------------------------+
i17 : L2 = differentialPrimaryDecomposition U;
i18 : netList L2
      +-----------------------------------+---------------------------------------------------------------------------------+
o18 = |ideal (x3, x1)                     |{| 1 |}                                                                          |
      |                                   | | 0 |                                                                           |
      +-----------------------------------+---------------------------------------------------------------------------------+
      |ideal (x4, x2)                     |{| 1 |}                                                                          |
      |                                   | | 0 |                                                                           |
      +-----------------------------------+---------------------------------------------------------------------------------+
      |ideal (x4, x3, x2)                 |{| -x1dx2 |, | -x1dx2^2 |, | -x1^2dx2^3-6x1dx2dx3 |, | -x1^2dx2^4-12x1dx2^2dx3 |}|
      |                                   | | 1      |  | 2dx2     |  | 3x1dx2^2+6dx3        |  | 4x1dx2^3+24dx2dx3       | |
      +-----------------------------------+---------------------------------------------------------------------------------+
      |ideal (x4, x3, x1)                 |{| -x2dx3 |, | -x2dx3^2 |, | -x2^2dx3^3-6x2dx3dx4 |, | -x2^2dx3^4-12x2dx3^2dx4 |}|
      |                                   | | 1      |  | 2dx3     |  | 3x2dx3^2+6dx4        |  | 4x2dx3^3+24dx3dx4       | |
      +-----------------------------------+---------------------------------------------------------------------------------+
      |ideal (x4, x2, x1)                 |{| -x3dx4 |, | -x3dx4^2 |, | -x3^2dx4^3-6x3dx1dx4 |, | -x3^2dx4^4-12x3dx1dx4^2 |}|
      |                                   | | 1      |  | 2dx4     |  | 3x3dx4^2+6dx1        |  | 4x3dx4^3+24dx1dx4       | |
      +-----------------------------------+---------------------------------------------------------------------------------+
      |ideal (x3, x2, x1)                 |{| -x4dx1 |, | -x4dx1^2 |, | -x4^2dx1^3-6x4dx1dx2 |, | -x4^2dx1^4-12x4dx1^2dx2 |}|
      |                                   | | 1      |  | 2dx1     |  | 3x4dx1^2+6dx2        |  | 4x4dx1^3+24dx1dx2       | |
      +-----------------------------------+---------------------------------------------------------------------------------+
      |ideal (x3 - x4, x2 - x4, x1 - x4)  |{| -1 |}                                                                         |
      |                                   | | 1  |                                                                          |
      +-----------------------------------+---------------------------------------------------------------------------------+
      |ideal (x3 + x4, x2 - x4, x1 + x4)  |{| 1 |}                                                                          |
      |                                   | | 1 |                                                                           |
      +-----------------------------------+---------------------------------------------------------------------------------+
      |                           2     2 |                                                                                 |
      |ideal (x2 + x4, x1 + x3, x3  + x4 )|{| -x3 |}                                                                        |
      |                                   | | x4  |                                                                         |
      +-----------------------------------+---------------------------------------------------------------------------------+
      |ideal (x4, x3, x2, x1)             |{| -2dx1dx2dx3^2-2dx1^2dx3dx4           |}                                       |
      |                                   | | 2dx1dx2^2dx3+dx1^2dx3^2+2dx1dx3dx4^2 |                                        |
      +-----------------------------------+---------------------------------------------------------------------------------+
i19 : U == intersect apply(L2, getModuleFromNoetherianOperators)
o19 = true
i20 : amult U
o20 = 22
\end{verbatim}
Again, $\amult(U) = 22$ is the size of the computed differential primary decomposition.
\end{example}

\section*{Acknowledgments}
We thank Rida Ait El Manssour, Marc H\"ark\"onen, Anton Leykin and Bernd Sturmfels for several joyful discussions on the topics appearing in this paper.
We are especially grateful to Marc H\"ark\"onen for his improvements to some of the \texttt{Macaulay2} implementations, and to Bernd Sturmfels for his constant guidance and support, throughout the duration of this project.
We thank the reviewer for their suggestions for the improvement of this work.


\end{document}